\documentclass[12pt]{amsart}
\usepackage{a4wide,color,graphicx}
\allowdisplaybreaks

\let\pa\partial  
\let\na\nabla  
\let\eps\varepsilon  
\newcommand{\N}{{\mathbb N}}  
\newcommand{\R}{{\mathbb R}} 
\newcommand{\diver}{\operatorname{div}}

\newtheorem{theorem}{Theorem}   
\newtheorem{lemma}[theorem]{Lemma}

\newtheorem{corollary}[theorem]{Corollary}

\begin{document}  

\title[Semiconductor energy-transport models]{Global existence analysis 
for degenerate energy-transport models for semiconductors}

\author{Nicola Zamponi}
\address{Institute for Analysis and Scientific Computing, Vienna University of  
	Technology, Wiedner Hauptstra\ss e 8--10, 1040 Wien, Austria}
\email{nicola.zamponi@tuwien.ac.at} 

\author{Ansgar J\"ungel}
\address{Institute for Analysis and Scientific Computing, Vienna University of  
	Technology, Wiedner Hauptstra\ss e 8--10, 1040 Wien, Austria}
\email{juengel@tuwien.ac.at} 

\date{\today}

\thanks{The authors acknowledge partial support from   
the Austrian Science Fund (FWF), grants P20214, P22108, I395, and W1245,
and from the Austrian-French Project Amad\'ee of the Austrian Exchange Service (\"OAD)} 

\begin{abstract}
A class of energy-transport equations without electric field 
under mixed Dirichlet-Neumann boundary conditions is analyzed.
The system of degenerate and strongly coupled parabolic equations for
the particle density and temperature arises in semiconductor device theory.
The global-in-time existence of weak nonnegative solutions is shown. 
The proof consists of a variable transformation
and a semi-discretization in time such that the discretized system becomes
elliptic and semilinear. Positive approximate solutions are obtained by
Stampacchia truncation arguments and a new cut-off test function.
Nonlogarithmic entropy inequalities yield gradient estimates which allow
for the limit of vanishing time step sizes. Exploiting the entropy inequality,
the long-time convergence of the weak solutions to the constant steady state
is proved. Because of the lack of appropriate convex Sobolev inequalities
to estimate the entropy dissipation, only an algebraic decay rate is obtained.
Numerical experiments indicate that the decay rate is typically exponential.
\end{abstract}

% \paragraph{Keywords:}  
 
% \paragraph{AMS classification:}  

\maketitle

%%%%%%%%%%%%%%%%%%%%%%%%%%%%%%%%%%%%%%%%%%%%%%%%%%%%%%%%%%%%%%%%%%%%%%%%%%%%%%%

\section{Introduction}

In this paper, we prove the global well-posedness of the
energy-transport equations
\begin{equation}
  \pa_t n = \Delta(n\theta^{1/2-\beta}), \quad
	\pa_t(n\theta) = \kappa\Delta(n\theta^{3/2-\beta}) + \frac{n}{\tau}(1-\theta)
	\quad\mbox{in }\Omega,\ t>0, \label{1.eq}
\end{equation}
where $-\frac12\le\beta<\frac12$, $\kappa=\frac23(2-\beta)$, and
$\Omega\subset\R^d$ with $d\le 3$ is a bounded domain.
This system describes the evolution of a fluid of particles with density $n(x,t)$
and temperature $\theta(x,t)$. The parameter $\tau>0$ is the relaxation time,
which is the typical time of the system to relax to the thermal equilibrium state
of constant temperature.
The system arises in the modeling of semiconductor devices
in which the elastic electron-phonon scattering is dominant. The above model is a 
simplification for vanishing electric fields. The full model was derived from
the semiconductor Boltzmann equation in the diffusion limit using a Chapman-Enskog
expansion around the equilibrium distribution \cite{BeDe96}. 
The parameter $\beta$ appears in the elastic scattering rate 
\cite[Section 6.2]{Jue09}. Certain values were used in the physical literarure,
for instance $\beta=\frac12$ \cite{CKRSD92}, $\beta=0$ \cite{LPSV92}, and 
$\beta=-\frac12$ \cite[Chapter 9]{Jue09}. The choice $\beta=\frac12$ leads 
in our situation to two uncoupled heat equations for $n$ and $n\theta$ 
and does not need to be considered.
We impose physically motivated mixed Dirichlet-Neumann boundary and 
initial conditions
\begin{align}
  & n=n_D,\ \theta=\theta_D\quad\mbox{on }\Gamma_D, \quad
	\na (n\theta^{1/2-\beta})\cdot\nu=\na(n\theta^{3/2-\beta})\cdot\nu=0
	\quad\mbox{on }\Gamma_N, \ t>0, \label{1.bc} \\
	& n(0)=n_0, \quad \theta(0)=\theta_0\quad\mbox{in }\Omega,
	\label{1.ic}
\end{align}
where $\Gamma_D$ models the contacts, $\Gamma_N=\pa\Omega\backslash\Gamma_D$
the union of insulating boundary segments, and $\nu$ is the exterior unit normal
to $\pa\Omega$ which is assumed to exist a.e.

The mathematical analysis of \eqref{1.eq}-\eqref{1.ic} is challenging
since the equations are not in the usual divergence form, they are strongly 
coupled, and they degenerate at $\theta=0$. 
The strong coupling makes impossible to apply maximum principle arguments in order
to conclude the nonnegativity of the temperature $\theta$.
On the other hand, this system possesses an interesting mathematical structure.
First, it can be written in ``symmetric'' form by introducing the so-called entropy
variables $w_1=\log(n/\theta^{3/2})$ and $w_2=-1/\theta$. Then, setting
$w=(w_1,w_2)^\top$ and $\rho=(n,\tfrac32 n\theta)^\top$, 
\eqref{1.eq} is formally equivalent to
$$
  \pa_t \rho = \diver(A(n,\theta)\na w) 
	+ \frac{1}{\tau}\begin{pmatrix} 0 \\ n(1-\theta) \end{pmatrix},
$$
where the diffusion matrix 
$$
  A(n,\theta) = n\theta^{1/2-\beta}\begin{pmatrix}
	1 & (2-\beta)\theta \\ (2-\beta)\theta & (3-\beta)(2-\beta)\theta^2
	\end{pmatrix}
$$
is symmetric and positive semi-definite. 
Second, system \eqref{1.eq} possesses the entropy (or free energy)
$$
  S[n(t),(n\theta)(t)] = \int_\Omega n\log\frac{n}{\theta^{3/2}}dx,
$$
which is nonincreasing along smooth solutions to \eqref{1.eq}. 
Even more entropy functionals exist; see \cite{JuKr12} and below.
However, they do not provide a lower bound for $\theta$ when $n$ vanishes.
We notice that both properties, the symmetrization via entropy variables and
the existence of an entropy, are strongly related \cite{DGJ97a,Jue09}. 

Equations \eqref{1.eq} resemble the diffusion equation 
$\pa_t w = \Delta(a(x,t)w)$, which was analyzed by Pierre and Schmitt
\cite{PiSc00}. By Pierre's duality estimate, an $L^2$ bound for
$\sqrt{a}w$ in terms of the $L^2$ norm of $\sqrt{a}$ has been derived.
In our situation, we obtain even $H^1$ estimates for $w=n$ and $w=n\theta$.

In spite of the above structure, there are only a few analytical results for 
\eqref{1.eq}-\eqref{1.ic}. In earlier works, drift-diffusion equations
with temperature-dependent mobilities but without temperature gradients 
\cite{Yin95} (also see \cite{WuXu06}) or nonisothermal systems containing
simplified thermodynamic forces \cite{AlXi94} have been studied.
Xu included temperature gradients in the model but he truncated the Joule heating
to allow for a maximum principle argument \cite{Xu09}. 
Later, existence results for the complete energy-transport equations 
(including electric fields) have been achieved, see \cite{FaIt01,Gri99} 
for stationary solutions near thermal equilibrium,
\cite{ChHs03,CHL04} for transient solutions close to equilibrium, and
\cite{DGJ97,DGJ98} for systems with uniformly positive definite diffusion matrices.
This assumption on the diffusion matrix avoids the degeneracy at $\theta=0$.
A degenerate energy-transport system was analyzed in \cite{JPR13}, but only
a simplified (stationary) temperature equation was studied.  
All these results give partial answers to the well-posedness problem only.
In this paper, we prove for the first time a global-in-time existence
result for any data and with physical transport coefficients.

Surprisingly, the above logarithmic entropy structure does not help.
Our key idea is to use the new
variables $u=n\theta^{1/2-\beta}$ and $v=n\theta^{3/2-\beta}$ and
nonlogarithmic entropy functionals. Then system \eqref{1.eq} becomes
$$
  \pa_t N(u,v) = \Delta u, \quad \pa_t E(u,v) = \kappa\Delta v + R(u,v),
$$
where $N(u,v)=u^{3/2-\beta}v^{\beta-1/2}$, $E(u,v)=u^{1/2-\beta}v^{\beta+1/2}$, 
and $R(u,v)=\tau^{-1}N(u,v)(1-v/u)$.
Discretizing this system by the implicit Euler method and employing 
the Stampacchia truncation method and a particular cut-off 
test function, we are able to prove the nonnegativity of $u$, $v$, and $\theta$.

In the following, we detail our main results and explain the ideas of the proofs.
Let $\pa\Omega\in C^1$, $\mbox{meas}(\Gamma_D)>0$, and $\Gamma_N$ is relatively open
in $\pa\Omega$. Furthermore, let
\begin{align}
  & n_D,\ \theta_D\in L^\infty(\Omega)\cap H^1(\Omega), \quad
	\inf_{\Gamma_D}n_D>0,\ \inf_{\Gamma_D}\theta_D>0, \label{hypo.D} \\
	& n_0,\ \theta_0\in L^\infty(\Omega)\cap H^1(\Omega), \quad
	\inf_{\Omega}n_0>0,\ \inf_{\Omega}\theta_0>0. \label{hypo.0}
\end{align}
We define the space $H_D^1(\Omega)$ as the closure of $C_0^\infty(\Omega\cup\Gamma_N)$
in the $H^1$ norm \cite[Section 1.7.2]{Tro87}. 
This space can be characterized by all functions
in $H^1(\Omega)$ which vanish on $\Gamma_D$ in the weak sense.
This space is the test function space for the weak formulation of \eqref{1.eq}.
Our first main result reads as follows.

\begin{theorem}[Global existence]\label{thm.ex}
Let $T>0$, $d\le 3$, $-\frac12\le\beta<\frac12$, $\tau>0$ 
and let \eqref{hypo.D}-\eqref{hypo.0}
hold. Then there exists a weak solution $(n,\theta)$ to \eqref{1.eq}-\eqref{1.ic}
such that $n\ge 0$, $n\theta\ge 0$ in $\Omega$, $t>0$, satisfying
\begin{align*}
  & n,\ n\theta,\ n\theta^{1/2-\beta},\ n\theta^{3/2-\beta}\in
	L^2(0,T;H^1(\Omega))\cap L^\infty(0,T;L^2(\Omega)), \\
  & \pa_t n,\ \pa_t(n\theta)\in L^2(0,T;H^1_D(\Omega)').
\end{align*}
\end{theorem}

The idea of the proof is to employ the implicit Euler method with time step $h>0$
and the new variables $u_j=n_j\theta_j^{1/2-\beta}$ and $v_j=n_j\theta_j^{3/2-\beta}$,
which approximate $u=n\theta^{1/2-\beta}$ and $v=n\theta^{3/2-\beta}$ 
at time $t_j=jh$, respectively. We wish to solve
\begin{equation}\label{1.disc}
  (n_j-n_{j-1}) - h\Delta u_j = 0, \quad
	\frac{1}{\kappa}(n_j\theta_j-n_{j-1}\theta_{j-1}) - h\Delta v_j 
	= \frac{hn_j}{\kappa\tau}(1-\theta).
\end{equation}
To simplify the presentation, we ignore the boundary conditions
and a necessary truncation of the temperature
(see Section \ref{sec.ex} for a full proof).  
A nice feature of this formulation is that we
can apply a Stampacchia truncation procedure to prove the strict positivity
of $u_j$ and $v_j$ (see Step 2 in the proof of Theorem \ref{thm.ex}).

The main difficulty is to show the positivity of $\theta_j=v_j/u_j$. 
We define a nondecreasing smooth cut-off function $\phi$ such that
$\phi(x)=0$ if $x\le M$ and $\phi(x)>0$ if $x>M$ for some $M>0$. 
We use the test functions
$u_j\phi(1/\theta_j)$ and $v_j\phi(1/\theta_j)$ in the weak formulation
of \eqref{1.disc}, respectively, and we subtract both equations to find after 
a straightforward computation (see Step 3 in the proof of Theorem \ref{thm.ex})
that
\begin{align*}
   0 &= \int_\Omega\bigg(\left(1-\frac{1}{\kappa}-\frac{h}{\kappa\tau}\right)n_jv_j
	\phi\left(\frac{1}{\theta_j}\right)
	+ \frac{v_j}{\kappa}n_{j-1}\theta_{j-1}
	\left(\frac{1}{\theta_j}-\frac{\kappa}{\theta_{j-1}}
	\right)\phi\left(\frac{1}{\theta_j}\right) \\
	&\phantom{xx}{}
	+ \frac{h}{v_j^2}\big|v_j\na u_j-u_j\na v_j\big|^2\phi'
	\left(\frac{1}{\theta_j}\right)
  + \frac{hn_j\theta_jv_j}{\kappa\tau}\phi\left(\frac{1}{\theta_j}\right)\bigg)dx.
\end{align*}
Since $\kappa>1$, there exists $h>0$ sufficiently small such that the first summand
becomes nonnegative. The third and last summands are nonnegative, too.
(Recall that we need to truncate $\theta_j$ with positive truncation.)
Hence, the integral over the second term is nonpositive. 
Then, choosing $M\ge\kappa/\theta_{j-1}$,
$$
  0\ge \int_\Omega v_jn _{j-1}\theta_{j-1}
	\left(\frac{1}{\theta_j}-\frac{\kappa}{\theta_{j-1}}
	\right)\phi\left(\frac{1}{\theta_j}\right)dx 
	\ge \int_\Omega v_jn _{j-1}\theta_{j-1}
	\left(\frac{1}{\theta_j}-M\right)\phi\left(\frac{1}{\theta_j}\right)dx.
$$
Because $\phi(1/\theta_j)=0$ for $1/\theta_j\le M$, this is only possible
if $1/\theta_j-M\le 0$ or $\theta_j\ge 1/M>0$. Clearly, 
the bound $M$ depends on $j$, and in the
de-regularization limit $h\to 0$, the limit of $\theta_j$ 
becomes nonnegative only.

A priori estimates which are uniform in the approximation parameter $h>0$
are obtained by proving a discrete version of the entropy inequality \cite{JuKr12}
\begin{equation}\label{1.ei}
  \frac{d}{dt}\int_\Omega n^2\theta^b dx + C_1\int_\Omega
	\big|\na\big(n\theta^{(2b+1-2\beta)/4}\big)\big|^2 dx \le C_2,
\end{equation}
for some $b\in\R$ and $C_1$, $C_2>0$. 
Choosing a variant of the sum of two entropies
$\int_\Omega n^2(\theta^{\beta-1/2}+\theta^5)dx$,
 we are able to derive gradient estimates for $n_j$,
$n_j\theta_j^{1/2-\beta}$, and $n_j\theta_j^{3/2-\beta}$ (see Step 4
of the proof of Theorem \ref{thm.ex}). Together with
Aubin's lemma and weak compactness arguments, the limit $h\to 0$ can be performed.

Theorem \ref{thm.ex} can be generalized in different ways. First, the boundary
data may depend on time. We do not consider this case here to avoid too many
technicalities. We refer to \cite{DGJ97} for the treatment of
time-dependent boundary functions. Second, we may allow for temperature-dependent
relaxation times,
\begin{equation}\label{1.tau}
  \tau(\theta)=\tau_0+\tau_1\theta^{1/2-\beta},
\end{equation}
where $\tau_0>0$ and $\tau_1>0$. This expression can be derived by using
an energy-dependent scattering rate \cite[Example 6.8]{Jue09}.
For this relaxation time, the conclusion of Theorem \ref{thm.ex} holds.

\begin{corollary}[Global existence]\label{coro.ex}
Let the assumptions of Theorem \ref{thm.ex} hold except that the relaxation time
is given by \eqref{1.tau}. Then there exists a weak solution to 
\eqref{1.eq}-\eqref{1.ic} with the properties stated in Theorem \ref{thm.ex}.
\end{corollary}

However, we have not been able to include electric fields in the model.
For instance, in this situation, the first equation in \eqref{1.eq} becomes
$$
  \pa_t n = \diver(\na(n\theta^{1/2-\beta})+n\theta^{-1/2-\beta}\na V),
$$
where $V(x,t)$ is the electric potential which is a given function or the solution
of the Poisson equation \cite{Jue09}. The problem is the treatment of the
drift term $n\theta^{-1/2-\beta}\na V$ for which the techniques developed for the
standard drift-diffusion model (see, e.g., \cite{GaGr86}) do not apply.

Our second main result concerns the long-time behavior of the solutions.

\begin{theorem}[Long-time behavior]\label{thm.long}
Let $d\le 3$, $0\le\beta<\frac12$, $\tau>0$, and 
$n_D=\mbox{const.}$, $\theta_D=1$.
Let $(n,\theta)$ be the weak solution constructed in Theorem \ref{thm.ex}.
Then there exist constants $C_1$, $C_2>0$, which depend only on
$\beta$, $n_D$, $n_0$, and $\theta_0$, such that for all $t>0$,
$$
  \|n(t)-n_D\|_{L^2(\Omega)}^2 + \|n(t)\theta(t)-n_D\|_{L^2(\Omega)}^2
	\le \frac{C_1}{1+C_2 t}.
$$
\end{theorem}

The proof of this theorem is based on discrete entropy inequality estimates.
The main difficulty is to bound the entropy dissipation. Usually,
this is done by employing a convex Sobolev inequality (e.g.\ the
logarithmic Sobolev or Beckner inequality). However, these tools are
not available for the cross-diffusion system at hand, and we need to employ
another technique. Our idea is to estimate the entropy dissipation by using
another entropy (choosing different values for $b$ in the discrete version
of \eqref{1.ei}). Denoting the discrete (nonlogarithmic) entropy at time $t_j$
by $S[n_j,n_j\theta_j]$, we arrive at the inequality
$$
  S[n_j,n_j\theta_j] - S[n_{j-1},n_{j-1}\theta_{j-1}]
	\le Ch S[n_j,n_j\theta_j]^2,
$$
where $C>0$ is independent of the time step size $h$. 
A discrete nonlinear Gronwall lemma
then shows that $S[n_j,n_j\theta_j]$ behaves like $1/(hj)=1/t_j$, and in the
limit $h\to 0$, we obtain the result.

The paper is organized as follows. We prove Theorem \ref{thm.ex} and Corollary
\ref{coro.ex} in Section \ref{sec.ex}. Section \ref{sec.long} is devoted
to the proof of Theorem \ref{thm.long}. The numerical results in one space
dimension presented in Section \ref{sec.num} indicate that the existence
of solutions still holds for $\beta<-\frac12$ and $\beta>\frac12$ and that
the solutions converge exponentially fast to the steady state.

%%%%%%%%%%%%%%%%%%%%%%%%%%%%%%%%%%%%%%%%%%%%%%%%%%%%%%%%%%%%%%%%%%%%%%%%%%%%%%%

\section{Global existence of solutions}\label{sec.ex}

We prove Theorem \ref{thm.ex} and Corollary \ref{coro.ex}.

{\em Step 1: Reformulation.}
Let $T>0$, $N\in\N$, and set $h=T/N$. We consider the semi-discrete equations
\begin{align}
  \frac{1}{h}(n_j-n_{j-1}) &= \Delta(n_j\theta_j^{1/2-\beta}), \quad j=1,\ldots,N, 
	\label{ex.nj} \\
	\frac{1}{h}(n_j\theta_j - n_{j-1}\theta_{j-1})
	&= \kappa\Delta(n_j\theta_j^{3/2-\beta}) + \frac{1}{\tau}n_j(1-\theta_j)
	\label{ex.tj}
\end{align}
with the boundary conditions \eqref{1.bc}.
The idea is to reformulate the elliptic equations in terms of the new variables
$$
  u_j = n_j\theta_j^{1/2-\beta}, \quad v_j = n_j\theta_j^{3/2-\beta}.
$$
Observing that $n_j=u_j^{3/2-\beta}v_j^{\beta-1/2}$ and $\theta_j=v_j/u_j$,
equations \eqref{ex.nj}-\eqref{ex.tj} are formally equivalent to
\begin{align}
  u_j^{3/2-\beta}v_j^{\beta-1/2} - h\Delta u_j 
	&= u_{j-1}^{3/2-\beta}v_{j-1}^{\beta-1/2}, \label{ex.uj} \\
	u_j^{1/2-\beta}v_j^{\beta+1/2} - \kappa h\Delta v_j
	- \frac{h}{\tau}u_j^{1/2-\beta}v_{j-1}^{\beta-1/2}(u_j-v_j)
	&= u_{j-1}^{1/2-\beta}v_{j-1}^{\beta+1/2}. \label{ex.vj}
\end{align}
The boundary conditions become
\begin{align}
  & u_j = u_D := n_D\theta_D^{1/2-\beta}, \quad 
	v_j = v_D := n_D\theta_D^{3/2-\beta}
	\quad\mbox{on }\Gamma_D, \label{ex.bc1} \\
	& \na u_j\cdot\nu = \na v_j\cdot\nu = 0\quad\mbox{on }\Gamma_N. \label{ex.bc2}
\end{align}
In order to show the existence of weak solutions to this discretized system,
we need to truncate. For this, let $j\ge 1$ and let
$u_{j-1}$, $v_{j-1}\in L^2(\Omega)$
be given such that $\inf_\Omega u_{j-1}>0$, $\inf_\Omega v_{j-1}>0$,
$\sup_\Omega u_{j-1}<+\infty$, and $\sup_\Omega v_{j-1}<+\infty$. We define
\begin{equation}\label{ex.M}
  M = \max\left\{\kappa\sup_\Omega\frac{u_{j-1}}{v_{j-1}},
	\frac{1}{\inf_{\Gamma_D}\theta_D}\right\}
\end{equation}
and $\eps=1/M$. The truncated problem reads as
\begin{align}
  u_j\theta_{j,\eps}^{\beta-1/2}-h\Delta u_j 
	&= u_{j-1}^{3/2-\beta}v_{j-1}^{\beta-1/2}, \label{ex.ujt} \\
	\left(1+\frac{h}{\tau}\right)v_j\theta_{j,\eps}^{\beta-1/2}
	- \kappa h\Delta v_j - \frac{h}{\tau}u_j\theta_{j,\eps}^{\beta-1/2}
	&= u_{j-1}^{1/2-\beta}v_{j-1}^{\beta+1/2}, \label{ex.vjt}
\end{align}
where $\theta_{j,\eps}=\max\{\eps,v_j/u_j\}$. Note that if $u_j>0$ and
$v_j/u_j\ge \eps$ in $\Omega$ then \eqref{ex.ujt}-\eqref{ex.vjt} are equivalent
to \eqref{ex.uj}-\eqref{ex.vj}. 

{\em Step 2: Solution of the truncated semi-discrete problem.}
We define the operator $F:L^2(\Omega)\times[0,1]\to L^2(\Omega)$ by
$F(\theta,\sigma)=v/u$, where $(u,v)\in H^1(\Omega)^2$ is the unique solution
to the linear system
\begin{align}
  \sigma u\theta_\eps^{\beta-1/2} - h\Delta u 
	&= \sigma u_{j-1}^{3/2-\beta}v_{j-1}^{\beta-1/2}
	= \sigma u_{j-1}\left(\frac{u_{j-1}}{v_{j-1}}\right)^{1/2-\beta}, 
	\label{ex.ut} \\
	\sigma\left(1+\frac{h}{\tau}\right)v\theta_\eps^{\beta-1/2} 
	- \kappa h\Delta v
	-\sigma\frac{h}{\tau} u\theta_\eps^{\beta-1/2}
	&= \sigma u_{j-1}^{1/2-\beta}v_{j-1}^{\beta+1/2}
	= \sigma v_{j-1}\left(\frac{u_{j-1}}{v_{j-1}}\right)^{1/2-\beta}, \label{ex.vt}
\end{align}
where $\theta_\eps=\max\{\eps,\theta\}$, with the boundary conditions 
\begin{equation}\label{ex.bct}
  u = 1+\sigma(u_D-1),\ v=\sigma v_D\quad\mbox{on }\Gamma_D, \quad
	\na u\cdot\nu = \na v\cdot\nu = 0 \quad\mbox{on }\Gamma_N.
\end{equation}
We have to prove that the operator $F$ is well defined.

First, observe that \eqref{ex.ut} does not depend on $v$ and that
the right-hand side is an element of $L^2(\Omega)$. Therefore, by
standard theory of elliptic equations, we infer the existence of a unique solution
$u\in H^1(\Omega)$ to \eqref{ex.ut} with the corresponding boundary
conditions in \eqref{ex.bct}. With given $u$, there exists a unique solution
$v\in H^1(\Omega)$ to \eqref{ex.vt} with the corresponding boundary conditions.
It remains to show that $u$ and $v$ are strictly positive in $\Omega$ such that
the quotient $v/u$ is defined and an element of $L^2(\Omega)$.

To this end, we employ the Stampacchia truncation method. Let
$$
  m_1 = \min\left\{\inf_{\Gamma_D}u_D,
	\eps^{1/2-\beta}\inf_\Omega u_{j-1}^{3/2-\beta}v_{j-1}^{\beta-1/2}\right\}> 0.
$$
Note that $m_1>0$ because of our boundedness assumptions on 
$\inf_\Omega u_{j-1}$ and $\sup_\Omega v_{j-1}$.
Then $(u-m_1)_-=\min\{0,u-m_1\}\in H_D^1(\Omega)$ is an admissible test function
in the weak formulation of \eqref{ex.ut} yielding
\begin{align*}
  h\int_\Omega & |\na(u-m_1)_-|^2 dx + \sigma\int_\Omega\theta_\eps^{\beta-1/2}
	(u-m_1)_-^2 dx \\
	&= \sigma\int_\Omega\big(u_{j-1}^{3/2-\beta}v_{j-1}^{\beta-1/2}
	- m_1\theta_\eps^{\beta-1/2}\big)(u-m_1)_- dx \\
	&\le \sigma\int_\Omega\big(u_{j-1}^{3/2-\beta}v_{j-1}^{\beta-1/2}
	- m_1\eps^{\beta-1/2}\big)(u-m_1)_- dx \le 0, 
\end{align*}
taking into account $\theta_\eps^{\beta-1/2}\le \eps^{\beta-1/2}$ (observe
that $\beta<1/2$) and the definition of $m_1$. This implies that 
$(u-m_1)_-=0$ and consequently $u\ge m_1>0$ in $\Omega$. Defining
$$
  m_2 = \min\left\{\inf_{\Gamma_D}v_D,
	\left(1+\frac{h}{\tau}\right)^{-1}
	\eps^{1/2-\beta}\inf_\Omega u_{j-1}^{1/2-\beta}v_{j-1}^{\beta+1/2}\right\}> 0
$$
and employing the test function $(v-m_2)_-\in H_D^1(\Omega)$ in the weak formulation
of \eqref{ex.vt}, a similar computation as above and 
$\theta_\eps^{\beta-1/2}\le\eps^{\beta-1/2}$ yield
\begin{align*}
  \kappa h\int_\Omega & |\na(v-m_2)_-|^2 dx
	+ \sigma\left(1+\frac{h}{\tau}\right)\int_\Omega\theta_\eps^{\beta-1/2}(v-m_2)_-^2 
	dx 
	- \frac{\sigma h}{\tau}\int_\Omega u\theta_\eps^{\beta-1/2}(v-m_2)_-dx \\
	&= \sigma\int_\Omega\left((u_{j-1}^{1/2-\beta}v_{j-1}^{\beta+1/2}
	- \left(1+\frac{h}{\tau}\right)m_2\theta_\eps^{\beta-1/2}\right)(v-m_2)_- dx
	\le 0.
\end{align*}
Since the integrals on the left-hand side are nonnegative, we conclude that
$v\ge m_2>0$ in $\Omega$. This shows that $u$ and $v$ are strictly positive with
a lower bound which depends on $\eps$ and $j$. Because of $1/u\in L^\infty(\Omega)$
and $u,v\in H^1(\Omega)\hookrightarrow L^6(\Omega)$, $v/u\in W^{1,3/2}(\Omega)
\hookrightarrow L^2(\Omega)$ for $d\le 3$. Hence, the operator $F$ is 
well defined and its image is contained in $W^{1,3/2}(\Omega)$.

Standard arguments and the compact embedding $W^{1,3/2}(\Omega)\hookrightarrow
L^2(\Omega)$ ensure that $F$ is continuous and compact. When $\sigma=0$,
it follows that $u=1$ and $v=0$ and thus, $F(\theta,0)=0$. 
Let $\theta\in L^2(\Omega)$ be a fixed point of $F(\cdot,\sigma)$. 
Then $v/u=\theta$. By standard elliptic estimates, 
we obtain $H^1$ bounds for $u$ and $v$ independently of $\sigma$.
Since $u$ is strictly positive, we infer an $L^2$ bound for $\theta$
independently of $\sigma$. Thus, we may apply the 
Leray-Schauder fixed-point theorem
to conclude the existence of a fixed point of $F(\cdot,1)$, i.e.\ of a solution
$(u,v)=(u_j,v_j)\in H^1(\Omega)^2$ to \eqref{ex.ujt}-\eqref{ex.vjt} 
with boundary conditions
\eqref{ex.bc1}-\eqref{ex.bc2}. 

In order to close the recursion, we need to show
that $\sup_\Omega u_j<+\infty$ and $\sup_\Omega v_j<+\infty$. 
We employ the following result which is due to Stampacchia \cite{Sta66}:
Let $w\in H^1(\Omega)$ be the unique solution to $-\Delta w+a(x)w=f$ 
with mixed Dirichlet-Neumann boundary conditions and let 
$a\in L^\infty(\Omega)$ be nonnegative and $f\in L^s(\Omega)$ with $s>d/2$.
Then $w\in L^\infty(\Omega)$ with a bound which depends only on $f$,
$\Omega$, and the boundary data. Since the right-hand side of \eqref{ex.ut}
is an element of $L^2(\Omega)$ and $d\le 3$, we find from the above
result that the solution
$u$ to \eqref{ex.ut} is bounded. Furthermore, $v$ solves (see \eqref{ex.vt})
$$
  \sigma\left(1+\frac{h}{\tau}\right)v\theta_\eps^{\beta-1/2}
	- \kappa h\Delta v = \sigma\frac{h}{\tau}u\theta_\eps^{\beta-1/2}
	+ \sigma u_{j-1}^{1/2-\beta}v_{j-1}^{\beta+1/2}\in L^\infty(\Omega),
$$
taking advantage of the $L^\infty$ bound for $u$. By Stampacchia's result,
$v\in L^\infty(\Omega)$. This shows the desired bounds.

{\em Step 3: Removing the truncation.}
We introduce the function
$$
  \phi(x) = \left\{\begin{array}{ll}
	0 & \quad\mbox{if }x\le M, \\
	1+\cos(\pi x/M) &\quad\mbox{if }M\le x\le 2M, \\
	2 &\quad\mbox{if }x\ge 2M,
	\end{array}\right.
$$
where we recall the definition \eqref{ex.M} of $M$. In particular,
$\phi\in C^1(\R)$ satisfies $\phi'\ge 0$ in $\R$.
Since $M\ge 1/\inf_{\Gamma_D}\theta_D$, we have
$\phi(u_j/v_j)=\phi(u_D/v_D)=\phi(1/\theta_D)=0$ on $\Gamma_D$. 
Because $\phi'$ vanishes outside of the interval $[M,2M]$, it holds that
$u_j\phi(u_j/v_j)$, $v_j\phi(u_j/v_j)\in H^1(\Omega)$.
Consequently, $v_j\phi(u_j/v_j)$ and $\kappa^{-1} u_j\phi(u_j/v_j)$
are admissible test functions in $H_D^1(\Omega)$
for \eqref{ex.ujt} and \eqref{ex.vjt},
respectively, which gives the two equations
\begin{align*}
  & \int_\Omega u_j\theta_{j,\eps}^{\beta-1/2}v_j\phi\left(\frac{u_j}{v_j}\right)dx
	+ h\int_\Omega\na u_j\cdot\na\left(v_j\phi\left(\frac{u_j}{v_j}\right)\right)dx
	= \int_\Omega u_{j-1}^{3/2-\beta}v_{j-1}^{\beta-1/2}v_j
	\phi\left(\frac{u_j}{v_j}\right)dx, \\
	& \frac{1}{\kappa}\left(1+\frac{h}{\tau}\right)\int_\Omega
	v_j\theta_{j,\eps}^{\beta-1/2}u_j\phi\left(\frac{u_j}{v_j}\right)dx
	+ h\int_\Omega\na v_j\cdot\na\left(u_j\phi\left(\frac{u_j}{v_j}\right)
	\right)dx \\
	&\phantom{xxxxxxx}{}-\frac{h}{\kappa\tau}\int_\Omega 
	u_j^2\theta_{j,\eps}^{\beta-1/2}
	\phi\left(\frac{u_j}{v_j}\right)dx
	= \frac{1}{\kappa}\int_\Omega u_{j-1}^{1/2-\beta}v_{j-1}^{\beta+1/2}
	u_j\phi\left(\frac{u_j}{v_j}\right) dx.
\end{align*}
We take the difference of these equations:
\begin{align}
  &\left(1-\frac{1}{\kappa}\left(1+\frac{h}{\tau}\right)\right)\int_\Omega
	u_jv_j\theta_{j,\eps}^{\beta-1/2}\phi\left(\frac{u_j}{v_j}\right)dx \nonumber \\
  &\phantom{xxx}{}
	+ h\int_\Omega(v_j\na u_j-u_j\na v_j)\cdot\na\phi\left(\frac{u_j}{v_j}\right)dx
	+ \frac{h}{\kappa\tau}\int_\Omega u_j^2\theta_{j,\eps}^{\beta-1/2}
	\phi\left(\frac{u_j}{v_j}\right)dx \nonumber \\
	&\phantom{xxx}{}
	+ \frac{1}{\kappa}\int_\Omega u_{j-1}^{1/2-\beta}v_{j-1}^{\beta+1/2}
	v_j\phi\left(\frac{u_j}{v_j}\right)\left(\frac{u_j}{v_j}
	-\kappa\frac{u_{j-1}}{v_{j-1}}\right)dx = 0. \label{ex.aux}
\end{align}
Since $\beta<1/2$, we have $\kappa=\frac23(2-\beta)>1$. Therefore, we can choose
$0<h<(\kappa-1)\tau$ which implies that $1-\kappa^{-1}(1+h/\tau)>0$,
and the first integral is nonnegative. The same conclusion holds for the
second integral in \eqref{ex.aux} since
$$
  (v_j\na u_j-u_j\na v_j)\cdot\na\phi\left(\frac{u_j}{v_j}\right)
	= \frac{1}{v_j^2}\phi'\left(\frac{u_j}{v_j}\right)
	|v_j\na u_j-u_j\na v_j|^2	\ge 0.
$$
Also the third integral in \eqref{ex.aux} is nonnegative.
Hence, the fourth integral is nonpositive, which can be equivalently written as
$$
  \int_\Omega u_{j-1}^{1/2-\beta}v_{j-1}^{\beta+1/2}
	v_j\phi\left(\frac{u_j}{v_j}\right)\left(\frac{u_j}{v_j}
	-M\right)dx
	\le \int_\Omega u_{j-1}^{1/2-\beta}v_{j-1}^{\beta+1/2}
	v_j\phi\left(\frac{u_j}{v_j}\right)\left(\kappa\frac{u_{j-1}}{v_{j-1}}
	-M\right)dx.
$$
Taking into account definition \eqref{ex.M} of $M$, we infer that
the integral on the right-hand side is nonpositive, which shows that
$$
  \int_\Omega u_{j-1}^{1/2-\beta}v_{j-1}^{\beta+1/2}
	v_j\phi\left(\frac{u_j}{v_j}\right)\left(\frac{u_j}{v_j}
	-M\right)_+dx = 0,
$$
where $z_+=\max\{0,z\}$ for $z\in\R$, employing $\phi(u_j/v_j)=0$
for $u_j/v_j\le M$. Now, $\phi(u_j/v_j)>0$ for $u_j/v_j>M$, and we conclude that
$(u_j/v_j-M)_+=0$ and $u_j/v_j\le M$ in $\Omega$. Since $\eps=1/M$,
this means that $v_j/u_j\ge\eps$ and $\theta_{j,\eps}=v_j/u_j$. 
Consequently, we have proven the existence of a weak solution $(v_j,u_j)$
to the discretized problem \eqref{ex.uj}-\eqref{ex.vj} with the boundary conditions
\eqref{ex.bc1}-\eqref{ex.bc2}, which also yields a weak solution
$(n_j,\theta_j)$ to \eqref{ex.nj}-\eqref{ex.tj} with the boundary conditions
\eqref{1.bc}. 

{\em Step 4: Entropy estimates.} Let $b\in\R$ and define the functional
\begin{equation}\label{ex.defphi}
  \phi_b[n,n\theta] = \int_\Omega\left(f_b(n,n\theta) - f_{b,D} 
	- \frac{\pa f_{b,D}}{\pa n}(n-n_D) - \frac{\pa f_{b,D}}{\pa (n\theta)}
	(n\theta-n_D\theta_D)\right)dx,
\end{equation}
where $f_b(n,n\theta) = n^{2-b}(n\theta)^b$ and we have employed the abbreviations
$$
  f_{b,D} = f_b(n_D,n_D\theta_D), \quad
	\frac{\pa f_{b,D}}{\pa n} = \frac{\pa f_b}{\pa n}(n_D,n_D\theta_D), \quad
	\frac{\pa f_{b,D}}{\pa(n\theta)} = \frac{\pa f_b}{\pa(n\theta)}(n_D,n_D\theta_D).
$$
The function $f_b$ is convex if $b\ge 2$ or $b\le 0$ since 
$\det D^2f_b(n,n\theta)=b(b-2)\theta^{2(\beta-1)}$ and
$\textnormal{tr} D^2 f_b(n,n\theta)=(b-1)(b-2)\theta^b + b(b-1)\theta^{b-2}$.
We wish to derive a priori estimates from the so-called entropy functionals
$$
  S_{b_1,b_2}[n,n\theta] = \frac{1}{|b_1|}\phi_{b_1}[n,n\theta]
	+ \frac{1}{|b_2|}\phi_{b_2}[n,n\theta].
$$
The parameters $(b_1,b_2)$ are chosen from the following set:
$$
  N_\beta = \big\{(b_1,b_2)\in\R^2:b_1,b_2\in N^*_\beta,\
	b_1\le b_2,\ b_1\le\beta-\tfrac12,\
	b_2\ge\tfrac52-\beta,\big\},
$$
where $N^*_\beta$ consists of all $b\in\R$ such that $(1-2\beta)b+6>0$ and
$$
  4(2\beta-1)b^3 + 4(4\beta^2-12\beta+11)b^2 + (8\beta^3-44\beta^2+70\beta-73)b
	- 6(2\beta-1)^2 > 0.
$$
The set of all $(\beta,b)$ such that $b\in N^*_\beta$ is illustrated in
Figure \ref{fig.Nstarbeta}. In particular, we have $b\ge 2$ or $b\le 0$ for all
$b\in N^*_\beta$ with $-\frac12<\beta<\frac12$. It is not difficult to check that
$(\beta-\frac12,5)\in N_\beta$ for all $-\frac12<\beta<\frac12$.

\begin{figure}[ht]
\includegraphics[width=70mm]{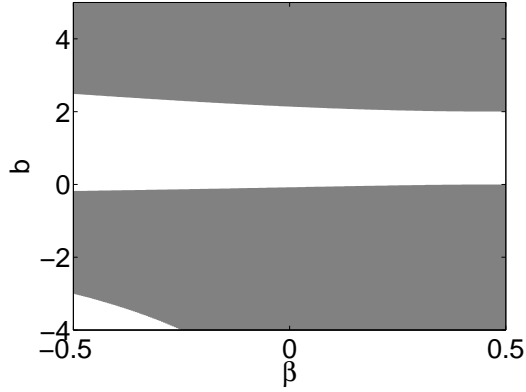}
\caption{The gray regions represent all points $(\beta,b)$ such that $b\in N^*_\beta$.}
\label{fig.Nstarbeta}
\end{figure}

\begin{lemma}[Discrete entropy inequality]\label{lem.S}
Let $(b_1,b_2)\in N_\beta$. Then
\begin{align}
  S_{b_1,b_2}[n_j,n_j & \theta_j] + C_1 h\int_\Omega\big(\theta_j^{b_1+1/2-\beta}
	+ \theta_j^{b_2+1/2-\beta}\big)|\na n_j|^2 dx \nonumber \\
	&\phantom{xx}{}+ C_1 h\int_\Omega n_j^2\big(\theta_j^{b_1-3/2-\beta}
	+ \theta_j^{b_2-3/2-\beta}\big)|\na\theta_j|^2 dx \nonumber \\
	&\le C_2 h + S_{b_1,b_2}[n_{j-1},n_{j-1}\theta_{j-1}], \label{ex.ei}
\end{align}
where $C_1>0$ depends on $b$ and $\beta$ and $C_2>0$ depends on $\tau$, $n_D$, 
and $\theta_D$. The constant $C_2$ vanishes if $n_D=\mbox{const}$.\ and $\theta_D=1$.
\end{lemma}

\begin{proof}
We abbreviate
$$
  f_{b,j} = f_b(n_j,n_j\theta_j), \quad
	\frac{\pa f_{b,j}}{\pa n} = \frac{\pa f_b}{\pa n}(n_j,n_j\theta_j), \quad
	\frac{\pa f_{b,j}}{\pa(n\theta)} = \frac{\pa f_b}{\pa(n\theta)}(n_j,n_j\theta_j).
$$
Let $b=b_1$ or $b=b_2$. We already observed that $b\ge 2$ or $b\le 0$. Hence,
$f_b(n,n\theta)$ is convex, and using \eqref{ex.nj}-\eqref{ex.tj}, we compute
\begin{align}
  \frac{1}{h}\big( & \phi_b[n_j,n_j\theta_j] 
	- \phi_b[n_{j-1},n_{j-1}\theta_{j-1}]\big) \nonumber \\
  &= \frac{1}{h}\int_\Omega\left((f_{b,j}-f_{b,j-1})
  - \frac{\pa f_{b,D}}{\pa n}(n_j-n_{j-1}) - \frac{\pa f_{b,D}}{\pa(n\theta)}
  (n_j\theta_j-n_{j-1}\theta_{j-1})\right)dx \nonumber \\
  &\le \frac{1}{h}\int_\Omega\bigg(\left(\frac{\pa f_{b,j}}{\pa n}
	-\frac{\pa f_{j,D}}{\pa n}\right)(n_j-n_{j-1}) 
	+ \left(\frac{\pa f_{j,b}}{\pa(n\theta)}-\frac{\pa f_{b,D}}{\pa(n\theta)}\right)
	(n_j\theta_j-n_{j-1}\theta_{j-1})\bigg)dx \nonumber \\
	&= -\int_\Omega\na\left(\frac{\pa f_{b,j}}{\pa n}
	-\frac{\pa f_{j,D}}{\pa n}\right)\cdot\na(n_j\theta_j^{1/2-\beta})dx \nonumber \\
	&\phantom{xx}{}- \kappa\int_\Omega\na \left(\frac{\pa f_{j,b}}{\pa(n\theta)}
	-\frac{\pa f_{b,D}}{\pa(n\theta)}\right)\cdot\na(n_j\theta_j^{3/2-\beta})dx 
	\nonumber \\
	&\phantom{xx}{}+ \frac{1}{\tau}\int_\Omega \left(\frac{\pa f_{j,b}}{\pa(n\theta)}
	-\frac{\pa f_{b,D}}{\pa(n\theta)}\right)n_j(1-\theta_j)dx. \label{ex.phi}
\end{align}
We estimate these integrals term by term.
First, we compute
\begin{align*}
  \int_\Omega & \left(\na\frac{\pa f_{b,j}}{\pa n}\cdot\na(n_j\theta_j^{1/2-\beta})
	+ \kappa\na\frac{\pa f_{j,b}}{\pa(n\theta)}
	\cdot\na(n_j\theta_j^{3/2-\beta})\right)dx \\
	&= \int_\Omega\big(A\theta_j^{b+1/2-\beta}|\na n_j|^2 
	+ 2Bn_j\theta_j^{b-1/2-\beta}\na n_j\cdot\na\theta_j
	+ Cn_j^2\theta_j^{b-3/2-\beta}|\na\theta_j|^2\big)dx,
\end{align*}
where, taking into account that $\kappa=\frac23(2-\beta)$,
\begin{align*}
  A &= \frac13(-2b\beta+b+6), \\
	B &= \frac{1}{12}(-2b\beta+b+6)(2b-2\beta+1), \\
	C &= \frac16 b(4b\beta^2-8b\beta-4\beta^2+9b+2\beta-6).
\end{align*}
The above integrand defines a quadratic form in $\theta_j^{(b+1/2-\beta)/2}
\na n_j$ and $n_j\theta_j^{(b-3/2-\beta)/2}\na\theta_j$
which is positive definite if and only if $A>0$ and $AC-B^2>0$. These two
conditions are equivalent to
\begin{align*}
  & (1-2\beta)b+6 > 0, \\
  & 4(2\beta-1)b^3 + 4(4\beta^2-12\beta+11)b^2 + (8\beta^3-44\beta^2+70\beta-73)b
	- 6(2\beta-1)^2 > 0,
\end{align*}
and these inequalities define the set $N^*_\beta$.
We infer that there exists a constant $C_1>0$ such that
\begin{align*}
  \int_\Omega & \left(\na\frac{\pa f_{b,j}}{\pa n}\cdot\na(n_j\theta_j^{1/2-\beta})
	+ \kappa\na\frac{\pa f_{j,b}}{\pa(n\theta)}
	\cdot\na(n_j\theta_j^{3/2-\beta})\right)dx \\
	&\phantom{xx}\ge C_1\int_\Omega\left(\theta_j^{b+1/2-\beta}|\na n_j|^2
	+ n_j^2\theta_j^{b-3/2-\beta}|\na\theta_j|^2\right)dx.
\end{align*}
The first two terms on the right-hand side of  \eqref{ex.phi}
involving the boundary contributions only are estimated by using
the Young inequality with $\delta>0$:
\begin{align*}
  \int_\Omega & \left( \na\frac{\pa f_{b,D}}{\pa n}\cdot
	\na(n_j\theta_j^{1/2-\beta})
	+ \kappa\na\frac{\pa f_{b,D}}{\pa(n\theta)}\cdot\na(n_j\theta_j^{3/2-\beta})
	\right)dx \\
	&\le \frac{1}{2\delta}\int_\Omega\left|\na\frac{\pa f_{b,D}}{\pa n}\right|^2 dx
	+ \frac{1}{2\delta}\int_\Omega\left|\na\frac{\pa f_{b,D}}{\pa(n\theta)}
	\right|^2 dx \\
	&\phantom{xx}{}+ \frac{\delta}{2}\int_\Omega\big(|\na(n_j\theta_j^{1/2-\beta})
	|^2
	+ \kappa^2|\na(n_j\theta_j^{3/2-\beta})|^2\big)dx \\
	&\le \frac{1}{2\delta}\int_\Omega\left|\na\frac{\pa f_{b,D}}{\pa n}\right|^2 dx
	+ \frac{1}{2\delta}\int_\Omega\left|\na\frac{\pa f_{b,D}}{\pa(n\theta)}\right|^2
	dx \\
	&\phantom{xx}{}
	+ C\delta\int_\Omega\big((\theta_j^{1-2\beta}+\theta_j^{3-2\beta})|\na n_j|^2
	+ n_j^2(\theta_j^{-1-2\beta}+\theta_j^{1-2\beta})|\na\theta_j|^2\big)dx,
\end{align*}
where $C>0$ depends only on $\beta$.
It remains to investigate the last integral in \eqref{ex.phi}
involving the relaxation term.
Since $\beta<1/2$, we have $b_1<0$ and $b_2>0$. Then
\begin{align*}
  \frac{1}{\tau}
  \sum_{b=b_1,b_2} & \frac{1}{|b|} \int_\Omega\left(\frac{\pa f_{j,b}}{
	\pa(n\theta)}
	-\frac{\pa f_{b,D}}{\pa(n\theta)}\right)n_j(1-\theta_j)dx \\
  &= \frac{1}{\tau}\sum_{b=b_1,b_2}\frac{b}{|b|}\int_\Omega
	(n_j\theta_j^{b-1}-n_D\theta_D^{b-1})n_j(1-\theta_j)dx \\
	&= -\frac{1}{\tau}\int_\Omega n_j^2\theta_j^{b_1-1}(\theta_j-1)
	(\theta_j^{b_2-b_1}-1)dx
	+ \frac{1}{\tau}\int_\Omega n_j n_D(\theta_j-1)
	(\theta_D^{b_2-1}-\theta_D^{b_1-1})dx.
\end{align*}
Since $b_1\le b_2$, the first expression on the right-hand side is nonpositive.
The second integral is written as
\begin{align*}
  \frac{1}{\tau}\int_\Omega n_j & n_D(\theta_j-1)(\theta_D^{b_2-1}-\theta_D^{b_1-1})dx \\
	&= \frac{1}{\tau}\int_\Omega \big((n_j\theta_j-n_D\theta_D)n_D - (n_j-n_D)n_D
	+ n_D^2(\theta_D-1)\big)(\theta_D^{b_2-1}-\theta_D^{b_1-1})dx \\
	&\le \int_\Omega g_D|n_j-n_D|dx + \int_\Omega g_D|n_j\theta_j-n_D\theta_D|dx
	+ \int_\Omega g^*_D dx,
\end{align*}
where the functions
$$
  g_D = \frac{n_D}{\tau}|\theta_D^{b_2-1}-\theta_D^{b_1-1}|, \quad
	g^*_D = n_D(\theta_D-1)g_D
$$
only depend on the boundary data.
Then the Young and Poincar\'e inequalities (with constant $C>0$) give
\begin{align*}
  \frac{1}{\tau}
  \sum_{b=b_1,b_2} & \frac{1}{|b|} \int_\Omega\left(\frac{\pa f_{j,b}}{
	\pa(n\theta)}
	-\frac{\pa f_{bD}}{\pa(n\theta)}\right)n_j(1-\theta_j)dx \\
	&\le \frac{\delta}{2}\int_\Omega|n_j-nD|^2 dx 
	+ \frac{\delta}{2}\int_\Omega|n_j\theta_j-n_D\theta_D|^2 dx
	+ \int_\Omega\left(g_D^*+\frac{1}{\delta}g_D^2\right)dx \\
	&\le C\frac{\delta}{2}\int_\Omega
	\big(|\na(n_j-n_D)|^2 + |\na(n_j\theta_j-n_D\theta_D)|^2\big)dx
	+ \int_\Omega\left(g_D^*+\frac{1}{\delta}g_D^2\right)dx \\
	&\le C\delta\int_\Omega\big(|\na n_j|^2 + \theta_j^2|\na n_j|^2
	+ n_j^2|\na\theta_j|^2\big)dx
	+ C\delta\int_\Omega\big(|\na n_D|^2 + |\na(n_D\theta_D|^2\big)dx \\
	&\phantom{xx}{}+ \int_\Omega\left(g_D^*+\frac{1}{\delta}g_D^2\right)dx.
\end{align*}
Putting together the above estimations and using 
$\theta_j^2|\na n_j|^2 \le C(1+\theta_j^{3-2\beta})|\na n_j|^2$, it follows that
\begin{align}
  \frac{1}{h} & \big(S_{b_1,b_2}[n_j,n_j\theta_j] 
	- S_{b_1,b_2}[n_{j-1},n_{j-1}\theta_{j-1}]\big) \nonumber \\
	&\phantom{xx}{}
	+ C_1\int_\Omega\big((\theta_j^{b_1+1/2-\beta}+\theta_j^{b_2+1/2-\beta})
	|\na n_j|^2
	+ n_j^2(\theta_j^{b_1-3/2-\beta}+\theta_j^{b_2-3/2-\beta})|\na\theta_j|^2\big)dx 
	\nonumber \\
	&\le C\delta\int_\Omega\big((1+\theta_j^{1-2\beta}+\theta_j^{3-2\beta})
	|\na n_j|^2
	+ n_j^2(1+\theta_j^{-1-2\beta}+\theta_j^{1-2\beta})|\na\theta_j|^2\big)dx 
	+ C_2, 
	\label{ex.aux2}
\end{align}
where the constant
$$
  C_2 = \frac{1}{2\delta}\int_\Omega
	\bigg(\left|\na\frac{\pa f_{b,D}}{\pa n}\right|^2 
	+ \left|\na\frac{\pa f_{b,D}}{\pa(n\theta)}\right|^2\bigg)dx
	+ \int_\Omega\left(g_D^*+\frac{1}{\delta}g_D^2\right)dx
$$
vanishes if $n_D=\mbox{const}$.\ and $\theta_D=1$.
The conditions $b_1\le\beta-1/2$ and $b_2\ge 5/2-\beta$ 
are equivalent to $b_1+1/2-\beta\le 0$ and $b_2+1/2-\beta\ge 3-2\beta$
as well as to $b_1-3/2-\beta\le -1-2\beta$ and $b_2-3/2-\beta\ge 1-2\beta$. 
Thus, there exists a positive constant $C>0$, which depends on $b_1$, $b_2$, 
and $\beta$, such that for all $\theta_j\ge 0$,
\begin{align*}
  1+\theta_j^{1-2\beta}+\theta_j^{3-2\beta}
	&\le C(\theta_j^{b_1+1/2-\beta}+\theta_j^{b_2+1/2-\beta}), \\
  1+\theta_j^{-1-2\beta}+\theta_j^{1-2\beta}
	&\le C(\theta_j^{b_1-3/2-\beta}+\theta_j^{b_2-3/2-\beta}).	
\end{align*}
Therefore, choosing $\delta>0$ sufficiently small, the integral on the right-hand
side of \eqref{ex.aux2} can be absorbed by the corresponding integral on
the left-hand side. This finishes the proof of the lemma.
\end{proof}

{\em Step 5: The limit $h\to 0$.} We define the piecewise constant
functions $n_h(x,t)=n_j(x)$ and $\theta_h(x,t)=\theta_j(x)$ for
$x\in\Omega$ and $t\in((j-1)h,jh]$, where $0\le j\le N=T/h$. 
The discrete time derivative of an arbitrary function $w(x,t)$
is defined by $(D_h w)(x,t)=h^{-1}(w(x,t)-w(x,t-h))$
for $x\in\Omega$, $t\ge h$. 
Then \eqref{ex.nj}-\eqref{ex.tj} can be written as
\begin{equation}
  D_h n_h = \Delta(n_h\theta_h^{1/2-\beta}), \quad
	D_h(n_h\theta_h) = \kappa\Delta(n_h\theta_h^{3/2-\beta}) 
	+ \frac{n_h}{\tau}(1-\theta_h). \label{ex.nh}
\end{equation}
The entropy inequality \eqref{ex.ei} for $(b_1,b_2)=(\beta-\frac12,5)\in N_\beta$
becomes, after summation over $j$,
\begin{align}
  S_{b_1,b_2}[n_h(t),n_h(t)\theta_h(t)] 
	&+ C_1\int_0^t\int_\Omega\big((1+\theta_h^{11/2-\beta})|\na n_h|^2
	+ n_h^2(\theta_h^{-2}+\theta_h^{7/2-\beta})|\na\theta_h|^2\big)dx\,ds 
	\nonumber \\
	&\le C_2 t + S_{b_1,b_2}[n_0,n_0\theta_0]. \label{ex.ei2}
\end{align}
We will exploit this inequality to derive $h$-independent estimates for
$(n_h)$ and $(n_h\theta_h)$.

\begin{lemma}\label{lem.est}
There exists a constant $C>0$ such that for all $h>0$,
\begin{align}
  \|n_h\|_{L^\infty(0,T;L^2(\Omega))} + \|n_h\theta_h\|_{L^\infty(0,T;L^2(\Omega))}
	&\le C, \label{ex.L2.1} \\
	\|n_h\theta_h^{1/2-\beta}\|_{L^\infty(0,T;L^2(\Omega))}
	+ \|n_h\theta_h^{3/2-\beta}\|_{L^\infty(0,T;L^2(\Omega))} &\le C, 
	\label{ex.L2.2} \\
	\|n_h\|_{L^2(0,T;H^1(\Omega))} + \|n_h\theta_h\|_{L^2(0,T;H^1(\Omega))} &\le C,
	\label{ex.H1.1} \\
	\|n_h\theta_h^{1/2-\beta}\|_{L^2(0,T;H^1(\Omega))}
	+ \|n_h\theta_h^{3/2-\beta}\|_{L^2(0,T;H^1(\Omega))} &\le C, \label{ex.H1.2} \\
  \|D_h n_h\|_{L^2(h,T;H^1_D(\Omega)')}
	+ \|D_h(n_h\theta_h)\|_{L^2(h,T;H^1_D(\Omega)')} &\le C. \label{ex.H-1.1}
\end{align}
%Moreover, if $n_D=\mbox{const}$.\ and $\theta_D=1$, the estimates can 
%be extended to the infinite time interval:
%\begin{align}
%  \|n_h-n_D\|_{L^2(0,\infty;H^1(\Omega))} 
%	+ \|n_h\theta_h-n_D\|_{L^2(0,\infty;H^1(\Omega))} &\le C, \label{ex.H1.3} \\
%	\|n_h\theta_h^{1/2-\beta}-n_D\|_{L^2(0,\infty;H^1(\Omega))}
%	+ \|n_h\theta_h^{3/2-\beta}-n_D\|_{L^2(0,\infty;H^1(\Omega))} &\le C, 
%	\label{ex.H1.4} \\
%	\|D_h n_h\|_{L^2(h,\infty;H^1_D(\Omega)')}
%	+ \|D_h(n_h\theta_h)\|_{L^2(h,\infty;H^1_D(\Omega)')} &\le C. \label{ex.H-1.2}
%\end{align}
\end{lemma}

\begin{proof}
First, we observe that there exists a constant $C>0$, which depends only on 
$\beta\in(-\frac12,\frac12)$, such that
\begin{align}
  1+\theta_h^2+\theta_h^{1-2\beta}+\theta_n^{3-2\beta}
	&\le C(\theta_h^{\beta-1/2}+\theta_h^5), \label{ex.th1} \\
	1+\theta_h^2+\theta_h^{1-2\beta}+\theta_n^{3-2\beta}
	&\le C(1+\theta_h^{11/2-\beta}), \label{ex.th2} \\
	1+\theta_h^{-1-2\beta}+\theta_h^{1-2\beta}
	&\le C(\theta_h^{-2}+\theta_h^{7/2-\beta}). \label{ex.th3}
\end{align}
We claim that for $(b_1,b_2)=(\beta-\frac12,5)$,
\begin{equation}\label{ex.aux3}
  S_{b_1,b_2}[n_h,n_h\theta_h] \ge -C + C\int_\Omega 
	n_h^2(\theta_h^{\beta-1/2}+\theta_h^5)dx,
\end{equation}
where $C>0$ is a (generic) constant independent of $h$.
Indeed, it holds $f_{b_1}(n_h,n_h\theta_h)=n_h^2\theta_h^{\beta-1/2}$
and $f_{b_2}(n_h,n_h\theta_h)=n_h^2\theta_h^{5}$, and the terms involving
the boundary data can be estimated according to 
\begin{align*}
  \sum_{b=b_1,b_2}\int_\Omega\left|\frac{\pa f_{b,D}}{\pa n}(n_h-n_D)\right|dx
	&\le C_\delta + \frac{\delta}{2}\int_\Omega n_h^2 dx, \\
	\sum_{b=b_1,b_2}\int_\Omega\left|\frac{\pa f_{b,D}}{\pa(n\theta)}
	(n_h\theta_h-n_D\theta_D)\right|dx
	&\le C_\delta + \frac{\delta}{2}\int_\Omega n_h^2\theta_h^2 dx,
\end{align*}
where we employed the Young inequality with $\delta>0$. 
We infer from \eqref{ex.th1} that
$$
  \frac{\delta}{2}\int_\Omega n_h^2(1+\theta_h^2)dx 
	\le \frac{\delta}{2}C \int_\Omega n_h^2(\theta_h^{\beta-1/2}+\theta_h^5)dx,
$$
and these terms can be absorbed for sufficiently small $\delta>0$ by the
corresponding terms coming from $f_{b_1}$ and $f_{b_2}$. 
This proves \eqref{ex.aux3}.
Now, we multiply \eqref{ex.th1} by $n_h^2$, integrate over $\Omega$, and employ
\eqref{ex.aux3}:
\begin{align*}
  \|n_h(t)\|_{L^2(\Omega)}^2 
	&+ \|n_h(t)\theta_h(t)\|_{L^2(\Omega)}^2  
	+ \|n_h(t)\theta_h(t)^{1/2-\beta}\|_{L^2(\Omega)}^2 \\  
	&+ \|n_h(t)\theta_h(t)^{3/2-\beta}\|_{L^2(\Omega)}^2 
	\le C\big(1+S_{b_1,b_2}[n_h(t),n_h(t)\theta_h(t)]\big).
\end{align*}
Taking into account the entropy inequality \eqref{ex.ei2}, 
estimates \eqref{ex.L2.1}-\eqref{ex.L2.2} follow.

Next, we compute, using \eqref{ex.th2}-\eqref{ex.th3},
\begin{align*}
  |\na n_h|^2 &+ |\na(n_h\theta_h)|^2 + |\na(n_h\theta_h^{1/2-\beta})|^2
	+ |\na(n_h\theta_h^{3/2-\beta})|^2 \\
	&\le C(1+\theta_h^2+\theta_h^{1-2\beta}+\theta_h^{3-2\beta})|\na n_h|^2
	+ Cn_h^2(1+\theta_h^{-1-2\beta}+\theta_h^{1-2\beta})|\na\theta_h|^2 \\
	&\le C(1+\theta_h^{11/2-\beta})|\na n_h|^2 
	+ Cn_h^2(\theta_h^{-2}+\theta_h^{7/2-\beta})|\na\theta_h|^2.
\end{align*}
Hence, Young's inequality gives
\begin{align*}
  |\na(n_h-&n_D)|^2 + |\na(n_h\theta_h-n_D\theta_D)|^2
	+ |\na(n_h\theta_h^{1/2-\beta}-n_D\theta_D^{1/2-\beta})|^2 \\
	&\phantom{xxxx}{}+ |\na(n_h\theta_h^{3/2-\beta}-n_D\theta_D^{3/2-\beta})|^2 \\
	&\phantom{xx}\le C(1+\theta_h^{11/2-\beta})|\na n_h|^2 
	+ Cn_h^2(\theta_h^{-2}+\theta_h^{7/2-\beta})|\na\theta_h|^2 + C_D,
\end{align*}
where $C_D>0$ depends on the $L^2$ norms of $\na n_D$, $\na(n_D\theta_D)$,
$\na(n_D\theta_D^{1/2-\beta})$, and $\na(n_D\theta_D^{3/2-\beta})$.
Note that $C_D=0$ if $n_D$ and $\theta_D$ are constant in $\Omega$.
We integrate over $\Omega\times(0,T)$ and employ the Poincar\'e inequality to find that
\begin{align*}
  \int_0^T&\big(\|n_h-n_D\|_{H^1(\Omega)}^2 
	+ \|n_h\theta_h-n_D\theta_D\|_{H^1(\Omega)}^2
	+ \|n_h\theta_h^{1/2-\beta}-n_D\theta_D^{1/2-\beta}\|_{H^1(\Omega)}^2 \\
	&\phantom{xx}{}
	+ \|n_h\theta_h^{3/2-\beta}-n_D\theta_D^{3/2-\beta}\|_{H^1(\Omega)}^2\big)dt \\
	&\le C\int_0^T\int_\Omega\big((1+\theta_h^{11/2-\beta})|\na n_h|^2
	+ n_h^2(\theta_h^{-2}+\theta_j^{7/2-\beta})|\na\theta_h|^2\big)dx\,dt + TC_D \\
	&\le C,
\end{align*}
because of the entropy inequality \eqref{ex.ei2}. This shows
\eqref{ex.H1.1}-\eqref{ex.H1.2}. 
%and, if $n_D$ and $\theta_D$ are constant,
%also \eqref{ex.H1.3}-\eqref{ex.H1.4}.

Finally, estimate \eqref{ex.H-1.1} follows from
\begin{align*}
  \|D_h n_h\|_{L^2(0,T;H^1_D(\Omega)')}
	&\le \|\na(n_h\theta_h^{1/2-\beta})\|_{L^2(0,T;L^2(\Omega))} \le C, \\
	\|D_h(n_h\theta_h)\|_{L^2(0,T;H^1_D(\Omega)')}
	&\le \kappa\|\na(n_h\theta_h^{3/2-\beta})\|_{L^2(0,T;L^2(\Omega))} \\
	&\phantom{xx}{}+ C\tau^{-1}\|n_h-n_h\theta_h\|_{L^2(0,T;L^2(\Omega))} \le C,
\end{align*}
using \eqref{ex.L2.1} and \eqref{ex.H1.2}. 
%If $n_D$ is constant and $\theta_D=1$,
%we can write for $\alpha\in\{0,1,1/2-\beta,3/2-\beta\}$,
%\begin{align*}
%  \|\na(n_h\theta_h^\alpha)\|_{L^2(0,T;L^2(\Omega))}
%	&= \|\na(n_h\theta_h^\alpha-n_D)\|_{L^2(0,T;L^2(\Omega))} \\
%	&\le \|n_h\theta_h^\alpha-n_D\|_{L^2(0,T;H^1(\Omega))},
%\end{align*}
%and the right-hand side is uniformly bounded by \eqref{ex.H1.3}-\eqref{ex.H1.4}.
%Together with the above estimates, the remaining bound \eqref{ex.H-1.2} follows.
\end{proof}

By weak compactness and the Aubin lemma in the version of \cite[Theorem 1]{DrJu12},
the uniform bounds of Lemma \ref{lem.est} imply the existence of subsequences
of $(n_h)$ and $(n_h\theta_h)$, which are not relabeled, such that as $h\to 0$,
\begin{align*}
  n_h\to n,\ n_h\theta_h \to w,\ n_h\theta_h^{1/2-\beta} \to y,\
	n_h\theta_h^{3/2-\beta} \to z &\quad\mbox{strongly in }L^2(0,T;L^2(\Omega)), \\
	n_h\rightharpoonup n,\ n_h\theta_h \rightharpoonup w,\ 
	n_h\theta_h^{1/2-\beta} \rightharpoonup y,\ 	
	n_h\theta_h^{3/2-\beta} \rightharpoonup z &\quad\mbox{weakly in }
	L^2(0,T;H^1(\Omega)), \\
	D_h n_h \rightharpoonup \pa_t n,\ D_h(n_h\theta_h)\rightharpoonup \pa_t w
	&\quad\mbox{weakly in }L^2(0,T;H_D^1(\Omega)').
\end{align*}
We wish to identify the limit functions $w$, $y$, and $z$. To this end,
we observe that (subsequences of) $n_h(t)$ and 
$n_h(t)\theta_h^{3/2-\beta}(t)$ converge
pointwise a.e.\ in $\Omega$ for a.e.\ $t\in(0,T)$. 
Hence, by Fatou's lemma and the Cauchy-Schwarz inequality,
\begin{align*}
  \int_\Omega & \frac{z(t)^{5/(3-2\beta)}}{n(t)^{(2+2\beta)/(3-2\beta)}}dx
	= \int_\Omega\liminf_{h\to 0}
	\frac{(n_h(t)\theta_h^{3/2-\beta})^{5/(3-2\beta)}}{
	n_h(t)^{(2+2\beta)/(3-2\beta)}}dx \\
	&\phantom{xx}\le \liminf_{h\to 0}\int_\Omega n_h(t)\theta_h(t)^{5/2}dx
	\le \mbox{meas}(\Omega)^{1/2}\liminf_{h\to 0}
	\left(\int_\Omega n_h(t)^2\theta_h(t)^5 dx\right)^{1/2} < \infty,
\end{align*}
since \eqref{ex.ei2} and \eqref{ex.aux3} show that the integral
of $n_h(t)^2\theta_h(t)^5$ is bounded uniformly in $h$ and $t$.
We infer that $z=0$ a.e.\ in $\{n=0\}$. We define $\theta=(z/n)^{2/(3-2\beta)}$
for $n>0$ and $\theta=0$ for $n=0$. Then $z=n\theta^{3/2-\beta}$. We pass
to the pointwise a.e.\ limit $h\to 0$ in
\begin{align*}
  n_h\theta_h^{1/2-\beta} &= n_h^{2/(3-2\beta)}
	\big(n_h\theta_h^{3/2-\beta}\big)^{(1-2\beta)/(3-2\beta)}, \\
  n_h\theta_h &= n_h^{(1-2\beta)/(3-2\beta)}
	\big(n_h\theta_h^{3/2-\beta}\big)^{2/(3-2\beta)},
\end{align*}
to find that, thanks to the pointwise a.e.\ convergence of the sequences
$(n_h)$, $(n_h\theta_h)$, $(n_h\theta^{1/2-\beta})$, and $(n_h\theta^{3/2-\beta})$,
\begin{align*}
  y &= n^{2/(3-2\beta)}\big(n\theta^{3/2-\beta}\big)^{(1-2\beta)/(3-2\beta)}
	= n\theta^{1/2-\beta}, \\
	w &= n^{(1-2\beta)/(3-2\beta)}
	\big(n\theta^{3/2-\beta}\big)^{2/(3-2\beta)} = n\theta.
\end{align*}
Now, we are in the position to perform the limit $h\to 0$ in \eqref{ex.nh}
which finishes the proof of Theorem \ref{thm.ex}.

{\em Step 6: Temperature-dependent relaxation times.} 
It remains to prove Corollary \ref{coro.ex}.
The proof is exactly as in Steps 1-5 except at two points. First,
we need to ensure in Step 2 that $\tau(\theta)$ is bounded from below
to obtain
$$
  \left(1-\frac{1}{\kappa}\left(1+\frac{h}{\tau(\theta_j)}\right)\right)
	\ge \left(1-\frac{1}{\kappa}\left(1+\frac{h}{\tau_0}\right)\right) > 0
$$
for all $0<h<(1-\kappa)\tau_0$,
which is needed to estimate \eqref{ex.aux}. Second, we need to pass to the limit
$h\to 0$ in the relaxation time term in Step 5. This is more involved since
we cannot perform the limit in $\tau(\theta_h)$. The idea is to expand the
fraction $n_h/\tau(\theta_h)$ and to consider
$$
  \frac{n_h}{\tau(\theta_h)}(1-\theta_h)
	= \frac{n_h^2(1-\theta_h)}{\tau_0 n_h+\tau_1 n_h\theta_h^{1/2-\beta}}.
$$
The pointwise convergences of $n_h\to n$ and 
$n_h\theta_h^{3/2-\beta}\to n\theta^{3/2-\beta}$ imply that
$$
  n_h\theta_h^{1/2} = n_h^{2(1-\beta)/(3-2\beta)}\big(n_h\theta_h^{3/2-\beta}
	\big)^{1/(3-2\beta)}
$$
converges pointwise to $n\theta^{1/2}$ as $h\to 0$. Consequently, we have
the pointwise convergence
$$
  \frac{n_h}{\tau(\theta_h)}(1-\theta_h)
	= \frac{n_h(n_h-n_h\theta_h)}{\tau_0 n_h+\tau_1 n_h\theta_h^{1/2-\beta}}
	\to \frac{n^2(1-\theta)}{\tau_0 n+\tau_1 n\theta^{1/2-\beta}}
	= \frac{n}{\tau(\theta)}(1-\theta).
$$
Furthermore, by \eqref{ex.L2.1},
$$
  \sup_{(0,T)}\int_\Omega \frac{n_h^2}{\tau(\theta_h)^2}(1-\theta_h)^2 dx
	\le \frac{2}{\tau_0^2}\sup_{(0,T)}\int_\Omega n_h^2(1+\theta_h^2) dx
$$
is uniformly bounded such that, together with the above pointwise convergence
and up to a subsequence, it follows that
$$
  \frac{n_h}{\tau(\theta_h)}(1-\theta_h)
  \rightharpoonup \frac{n}{\tau(\theta)}(1-\theta)\quad
	\mbox{weakly in }L^2(0,T;L^2(\Omega)).
$$
This ends the proof.

%%%%%%%%%%%%%%%%%%%%%%%%%%%%%%%%%%%%%%%%%%%%%%%%%%%%%%%%%%%%%%%%%%%%%%%%%%%%%%%

\section{Long-time behavior of solutions}\label{sec.long}

We prove Theorem \ref{thm.long}. The proof is divided into several steps.

{\em Step 1:} Let $(n_j,n_j\theta_j)$ be a solution
to \eqref{ex.nh} with boundary conditions \eqref{1.bc}. We recall that
both $n_j$ and $\theta_j$ are strictly positive. 
% We define the functionals
% \begin{align*}
%   \widetilde\phi_b[n,n\theta] &= \int_\Omega\left(f_b(n,n\theta)-f_b(n_D,n_D)
%   - \frac{\pa f_{b,D}}{\pa n}(n-n_D) - \frac{\pa f_{b,D}}{\pa(n\theta)}
%   (n\theta-n_D)\right)dx, \\
%   \widetilde S_{b_1,b_2}[n,n\theta] 
% 	&= \frac{1}{|b_1|}\widetilde\phi_{b_1}[n,n\theta]
% 	+ \frac{1}{|b_2|}\widetilde\phi_{b_2}[n,n\theta].
% \end{align*}
% The difference to definition \eqref{ex.defphi} of $\phi_b[n,n\theta]$
% is the additional term $f_b(n_D,n_D)$ which guarantees that
% $\widetilde\phi_b[n_D,n_D]=0$. The entropy inequality \eqref{ex.ei} still
% holds for $\widetilde S_{b_1,b_2}$. 
Observing that the constant 
boundary data gives $C_2=0$ in \eqref{ex.ei}, 
we obtain for $(b_1,b_2)\in N_\beta^*$,
\begin{align*}
  S_{b_1,b_2}[n_j,n_j\theta_j]
	&+ C_1 h\int_\Omega\big((\theta_j^{b_1+1/2-\beta}
	+\theta_j^{b_2+1/2-\beta})|\na n_j|^2 \\
	&+ n_j^2(\theta_j^{b_1-3/2-\beta}+\theta_j^{b_2-3/2-\beta})|\na\theta_j|^2
	\big)dx \le S_{b_1,b_2}[n_{j-1},n_{j-1}\theta_{j-1}].
\end{align*}
In particular, for $(b_1,b_2)=(\beta-1/2,5/2-\beta)\in N_\beta^*$,
\begin{align}
  S_{b_1,b_2}[n_j,n_j\theta_j]
	&+ C_1 h\int_\Omega\big(1+\theta_j^{3-2\beta})|\na n_j|^2 
	+ n_j^2(\theta_j^{-2}+\theta_j^{1-2\beta})|\na\theta_j|^2\big)dx \nonumber \\
	&\le S_{b_1,b_2}[n_{j-1},n_{j-1}\theta_{j-1}], \label{lo.beta-12}
\end{align}
and for $(b_1,b_2)=(-3,5)\in N_\beta^*$ (here, we need $\beta\ge 0$),
\begin{align}
  S_{b_1,b_2}[n_j,n_j\theta_j]
	&+ C_1 h\int_\Omega\big((\theta_j^{-5/2-\beta}+\theta_j^{11/2-\beta})
	|\na n_j|^2 \nonumber \\
	&+ n_j^2(\theta_j^{-9/2-\beta}+\theta_j^{7/2-\beta})|\na\theta_j|^2
	\big)dx \le S_{b_1,b_2}[n_{j-1},n_{j-1}\theta_{j-1}]. \label{lo.-3}
\end{align}

{\em Step 2:}
We show that the integral involving the gradient terms in \eqref{lo.beta-12}
can be bounded from below by, up to a factor, the entropy $S_{b_1,b_2}$. 
To this end, we observe that, by the convexity of 
$f_b(n,n\theta)=n^{2-b}(n\theta)^b$ for $b\le 0$ or $b\ge 2$,
\begin{align*}
  \int_\Omega\big( & f_b(n_j,n_j\theta_j)-f_b(n_D,n_D)\big)dx \\
	&\le \int_\Omega\left(\frac{\pa f_b}{\pa n}(n_j,n_j\theta_j)(n_j-n_D)
	+ \frac{\pa f_b}{\pa(n\theta)}(n_j,n_j\theta_j)(n_j\theta_j-n_D)\right)dx.
\end{align*}
This implies that
\begin{align*}
  \phi_b[n_j,n_j\theta_j] 
	&\le \int_\Omega\bigg(\left(\frac{\pa f_{b,j}}{\pa n}-\frac{\pa f_{j,D}}{\pa n}
	\right)(n_j-n_D) \\
	&\phantom{xx}{}+ \left(\frac{\pa f_{b,j}}{\pa(n\theta)}
	-\frac{\pa f_{j,D}}{\pa(n\theta)}\right)(n_j\theta_j-n_D)\bigg)dx \\
	&= \int_\Omega\big((2-b)(n_j\theta_j^b-n_D)(n_j-n_D)
	+ b(n_j\theta_j^{b-1}-n_D)(n_j\theta_j-n_D)\big)dx \\
	&\le C\big(1+\|n_j\theta_j^b\|_{L^2(\Omega)}\big)\|n_j-n_D\|_{L^2(\Omega)} \\
	&\phantom{xx}{}+ C\big(1+\|n_j\theta_j^{b-1}\|_{L^2(\Omega)}\big)
	\|n_j\theta_j-n_D\|_{L^2(\Omega)}.
\end{align*}
Hence, we obtain for $(b_1,b_2)=(\beta-1/2,5/2-\beta)$,
\begin{align}
  S_{b_1,b_2}&[n_j,n_j\theta_j]
	\le C\phi_{b_1}[n_j,n_j\theta_j]
	+ C\phi_{b_2}[n_j,n_j\theta_j] \nonumber \\
	&\le C\big(1+\|n_j\theta_j^{\beta-1/2}\|_{L^2(\Omega)}
	+\|n_j\theta_j^{5/2-\beta}\|_{L^2(\Omega)}\big)\|n_j-n_D\|_{L^2(\Omega)} 
	\nonumber \\
	&\phantom{xx}{}+ C\big(1+\|n_j\theta_j^{\beta-3/2}\|_{L^2(\Omega)}
	+\|n_j\theta_j^{3/2-\beta}\|_{L^2(\Omega)}\big)
	\|n_j\theta_j-n_D\|_{L^2(\Omega)}. \label{lo.aux3}
\end{align}
Noting that (again using $\beta\ge 0$)
$$
  n_j^2\theta_j^{2\beta-1} + n_j^2\theta_j^{2\beta-3} + n_j^2\theta_j^{3-2\beta}
	+ n_j^2\theta_j^{5-2\beta}
	\le Cn_j^2(\theta_j^{-3}+\theta_j^5)
$$
for some generic constant $C>0$ not depending on $h$,
we infer from \eqref{lo.-3}, after summation over $j$, that
$$
  \|n_j\theta_j^{\beta-1/2}\|_{L^2(\Omega)}
	+ \|n_j\theta_j^{5/2-\beta}\|_{L^2(\Omega)} \le C, \quad
	\|n_j\theta_j^{\beta-3/2}\|_{L^2(\Omega)}
	+ \|n_j\theta_j^{3/2-\beta}\|_{L^2(\Omega)} \le C,
$$
and $C>0$ does not depend on $j$ or $h$.
Thus, \eqref{lo.aux3} becomes, with $(b_1,b_2)=(\beta-1/2,5/2-\beta)$,
$$
  S_{b_1,b_2}[n_j,n_j\theta_j]
	\le C\|n_j-n_D\|_{L^2(\Omega)} + C\|n_j\theta_j-n_D\|_{L^2(\Omega)}.
$$
Taking the square and employing the Poincar\'e inequality yields
\begin{align*}
  S_{b_1,b_2}&[n_j,n_j\theta_j]^2 
	\le C\int_\Omega\big(|n_j-n_D|^2 + |n_j\theta_j-n_D|^2\big)dx \\
	&\le C\int_\Omega\big(|\na n_j|^2 + |\na(n_j\theta_j)|^2\big)dx 
	\le C\int_\Omega\big((1+\theta_j^2)|\na n_j|^2 + n_j^2|\na\theta_j|^2
	\big)dx \\
	&\le C\int_\Omega\big((1+\theta_h^{3-2\beta})|\na n_h|^2
	+ n_h^2(\theta_h^{-2}+\theta_h^{1-2\beta})|\na\theta_h|^2\big)dx.
\end{align*}
The last inequality follows from elementary estimations using the fact that
$\beta<1/2$. This is the desired estimate. 

Thus, it follows from \eqref{lo.beta-12} that
$$
  S_{b_1,b_2}[n_j,n_j\theta_j]
	+ Ch S_{b_1,b_2}[n_j,n_j\theta_j]^2 
	\le S_{b_1,b_2}[n_{j-1},n_{j-1}\theta_{j-1}],
$$
where still $(b_1,b_2)=(\beta-1/2,5/2-\beta)$.
We employ the following lemma which is a consequence of Lemma 17 in \cite{CJS13}.

\begin{lemma}
Let $(x_j)$ be a sequence of nonnegative numbers such that
$x_j+\kappa x_j^2 \le x_{j-1}$ for $j\in\N$. Then 
$$
  x_j \le \frac{x_0}{1+\kappa x_0 j/(1+2\kappa x_0)}, \quad j\in\N.
$$
\end{lemma}

Hence, with $S_0= S_{b_1,b_2}[n_0,n_0\theta_0]$,
$$
  S_{b_1,b_2}[n_j,n_j\theta_j]
	\le \frac{S_0}{1 + Chj S_0/(1+2Ch S_0)},
$$
which can be written as
\begin{equation}\label{lo.ineq}
  S_{b_1,b_2}[n_h(t),n_h(t)\theta_h(t)]
	\le \frac{S_0}{1 + Ct S_0/(1+2Ch S_0)}, \quad t>0.
\end{equation}

{\em Step 3:} It remains to prove a lower bound for $ S_{b_1,b_2}$.
We employ again the convexity of $f_b$:
\begin{equation}\label{lo.aux}
  S_{b_1,b_2}[n_h,n_h\theta_h]
	\ge C\int_\Omega\lambda\big(|n_h-n_D|^2 + |n_h\theta_h-n_D|^2\big)dx,
\end{equation}
where $\lambda$ is the minimal eigenvalue of the Hessian
$D^2 f_{b_1}(\xi_1,\xi_2)+D^2 f_{b_2}(\xi_1,\xi_2)$ and
$\xi_1 = \alpha n_h+(1-\alpha)n_D$, $\xi_2=\alpha n_h\theta_h+(1-\alpha)n_D$
for some $0\le\alpha\le 1$. 

We recall the following results from linear algebra. 
If $A$ and $B$ are two symmetric
matrices in $\R^{2\times 2}$ with minimal eigenvalues $\lambda_{\rm min}(A)$ and 
$\lambda_{\rm min}(B)$, respectively, then 
$\lambda_{\rm min}(A+B)\ge \lambda_{\rm min}(A)+\lambda_{\rm min}(B)$
(since the minimal eigenvalue is the minimum of the Rayleigh quotient).
Furthermore, a simple computation shows that
$\lambda_{\rm min}(A) = \frac12\mbox{tr}(A)-(\frac14\mbox{tr}(A)^2
-\det(A))^{1/2}\ge \det(A)/\mbox{tr}(A)$.
Consequently, since
\begin{align*}
  \det(D^2 f_b(\xi_1,\xi_2)) &= b(b-2)\eta^{2b-2}, \\
	\mbox{tr}(D^2 f_b(\xi_1,\xi_2)) &= (b-1)((b-2)\eta^b + b\eta^{b-2}),
\end{align*}
with $\eta=\xi_2/\xi_1$, we conclude that
\begin{align*}
  \lambda &\ge \lambda_{\rm min}(D^2 f_{b_1}(\xi_1,\xi_2))
	+ \lambda_{\rm min}(D^2 f_{b_2}(\xi_1,\xi_2)) \\
	&\ge \frac{\det(D^2 f_{b_1}(\xi_1,\xi_2))}{\mbox{tr}
	(D^2 f_{b_1}(\xi_1,\xi_2))}
	+ \frac{\det(D^2 f_{b_2}(\xi_1,\xi_2))}{\mbox{tr}
	(D^2 f_{b_2}(\xi_1,\xi_2))}
	\ge C\frac{\eta^{\beta-1/2}+\eta^{5/2-\beta}}{1+\eta^2}.
\end{align*}
Since $\beta<1/2$, the function $x\mapsto (x^{\beta-1/2}+x^{5/2-\beta})/(1+x^2)$
has a positive lower bound. Therefore, $\lambda$ is strictly positive
independent of $h$. Going back to \eqref{lo.aux}, we infer the lower bound
$$
  S_{b_1,b_2}[n_h,n_h\theta_h] 
	\ge C\int_\Omega\big(|n_h-n_D|^2 + |n_h\theta_h-n_D|^2\big)dx.
$$
Together with \eqref{lo.ineq}, this shows that
$$
  \|n_h(t)-n_D\|_{L^2(\Omega)}^2 + \|n_h(t)\theta_h(t)-n_D\|_{L^2(\Omega)}^2
	\le \frac{S_0}{1+C(S_0)t}, \quad t>0.
$$
In view of Lemma \ref{lem.est}, the sequences $(n_h)$ and $(n_h\theta_h)$
are bounded in $L^\infty(0,T;L^2(\Omega))$. Therefore, by Fatou's lemma,
we obtain
$$
  \|n(t)-n_D\|_{L^2(\Omega)}^2 + \|n(t)\theta(t)-n_D\|_{L^2(\Omega)}^2
	\le \frac{S_0}{1+C(S_0)t}, \quad t>0,
$$
which concludes the proof.

%%%%%%%%%%%%%%%%%%%%%%%%%%%%%%%%%%%%%%%%%%%%%%%%%%%%%%%%%%%%%%%%%%%%%%%%%%%%%

\section{Numerical experiments}\label{sec.num}

In this section we present some numerical results related to \eqref{1.eq}.
According to Theorem \ref{thm.long}, the solution $(n(t),n(t)\theta(t))$
converges to $(n_D,n_D\theta_D)$ in $L^2(\Omega)$ as $t\to\infty$ if
$n_D$, $\theta_D$ are constants and $\theta_D=1$. We want to check this 
behavior in the numerical simulations if the particle density and temperature are 
close to zero in some point initially.
 
We consider system \eqref{1.eq} in one space dimension with
$\Omega = (0,1)\subset\R$, and we impose Dirichlet boundary conditions 
at $x=0,1$ and initial conditions \eqref{1.ic}. We choose 
the boundary data $n_D=\theta_D=1$ and the initial
functions $n_0(x)=\exp(-48x^2)$ for $0\le x\le \frac12$,
$n_0(x)=\exp(-48(x-1)^2)$ for $\frac12<x\le 1$, and $\theta_0=n_0$.
Both initial functions are very small at $x=\frac12$; it holds
$n_0(\frac12)=\theta_0(\frac12)=\exp(-12)\approx 6.1\cdot 10^{-6}$.

The equations are discretized
in time by the implicit Euler method with time step $\triangle t$ and in 
space by central finite differences with space step $\triangle x$. 
The discretized nonlinear system is solved by the Newton method. 
The time step is chosen in an adaptive way: 
It is multiplied by the factor 1.25 when the initial guess in the Newton iterations 
satisfies already the tolerance imposed on the residual, 
and it is multiplied by the factor 0.75 
when the solution of the Newton system is not feasible (namely, not positive).
The space step is chosen as $\triangle x=2\cdot 10^{-3}$ (501 grid points)
and the maximal time step is $\triangle t=2\cdot 10^{-3}$.

Figures \ref{fig.betaminus025} and \ref{fig.betaplus025} illustrate the 
temporal behavior of the partical density $n$ and the temperature $\theta$
for $\beta=-0.25$ and $\beta=0.25$, respectively,
at various small times. 
For larger times, the functions approach the constant steady state. 
The diffusion causes the singularity at $x=\frac12$ to smooth out quickly, and
the solution converges to the steady state. In Figure
\ref{fig.longtime}, the decay of the relative $\ell^2$ difference to the
steady state is illustrated in a semi-logarithmic plot. Even for $\beta<-\frac12$
or $\beta>\frac12$, the decay to equilibrium seems to be exponentially fast,
at least after an initial phase. 
This may indicate that the decay rate of Theorem \ref{thm.long}
is not optimal. Moreover, the results indicate that there may exist solutions
to \eqref{1.eq}-\eqref{1.ic} even for $\beta<-\frac12$ and $\beta>\frac12$.

\begin{figure}[ht]
\includegraphics[width=80mm,height=60mm]{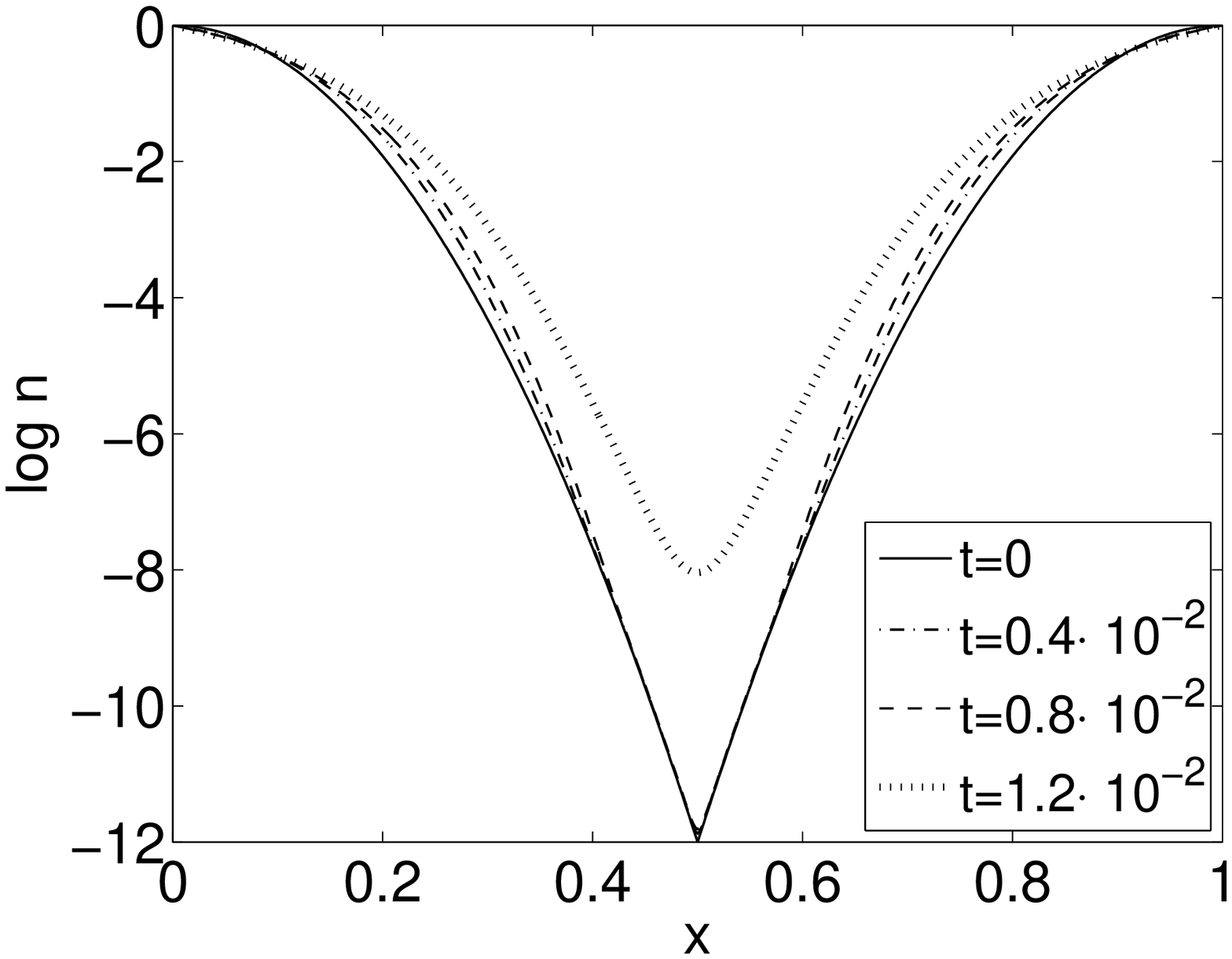}
\includegraphics[width=80mm,height=60mm]{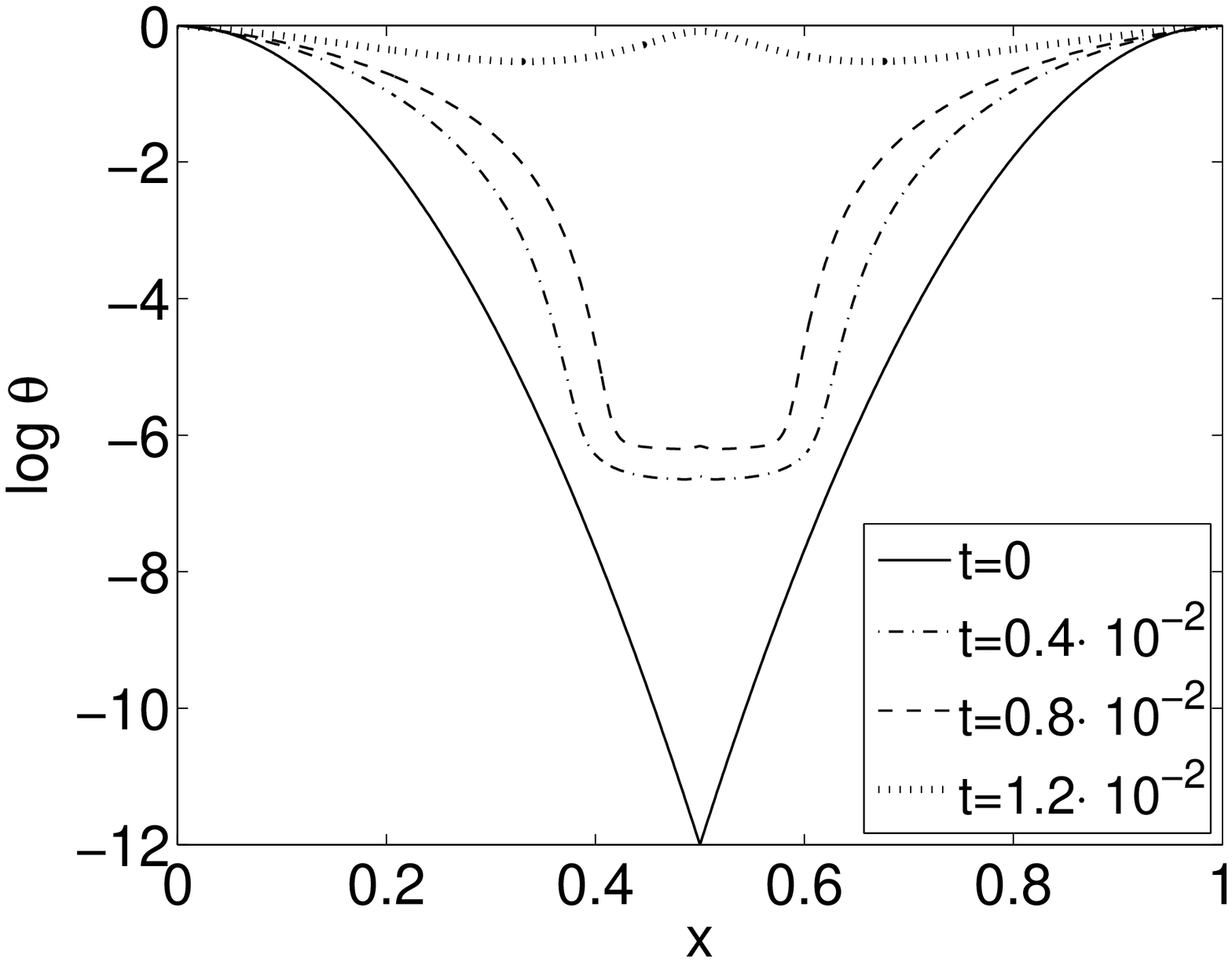}
\caption{Evolution of the particle density $n$ and the temperature $\theta$
(semi-logarithmic plot) at various times for $\beta=-0.25$.}
\label{fig.betaminus025}
\end{figure}

\begin{figure}[ht]
\includegraphics[width=80mm,height=60mm]{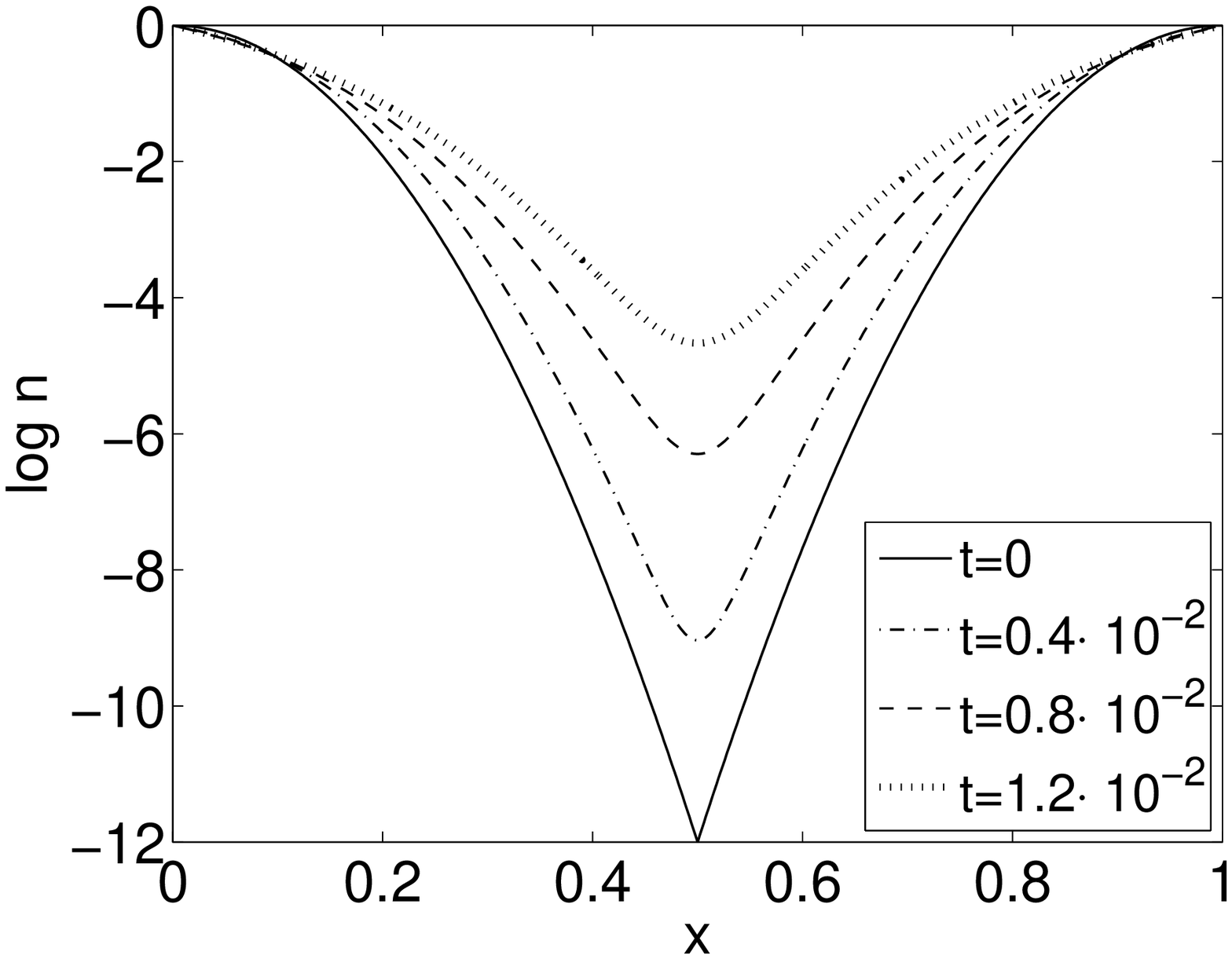}
\includegraphics[width=80mm,height=60mm]{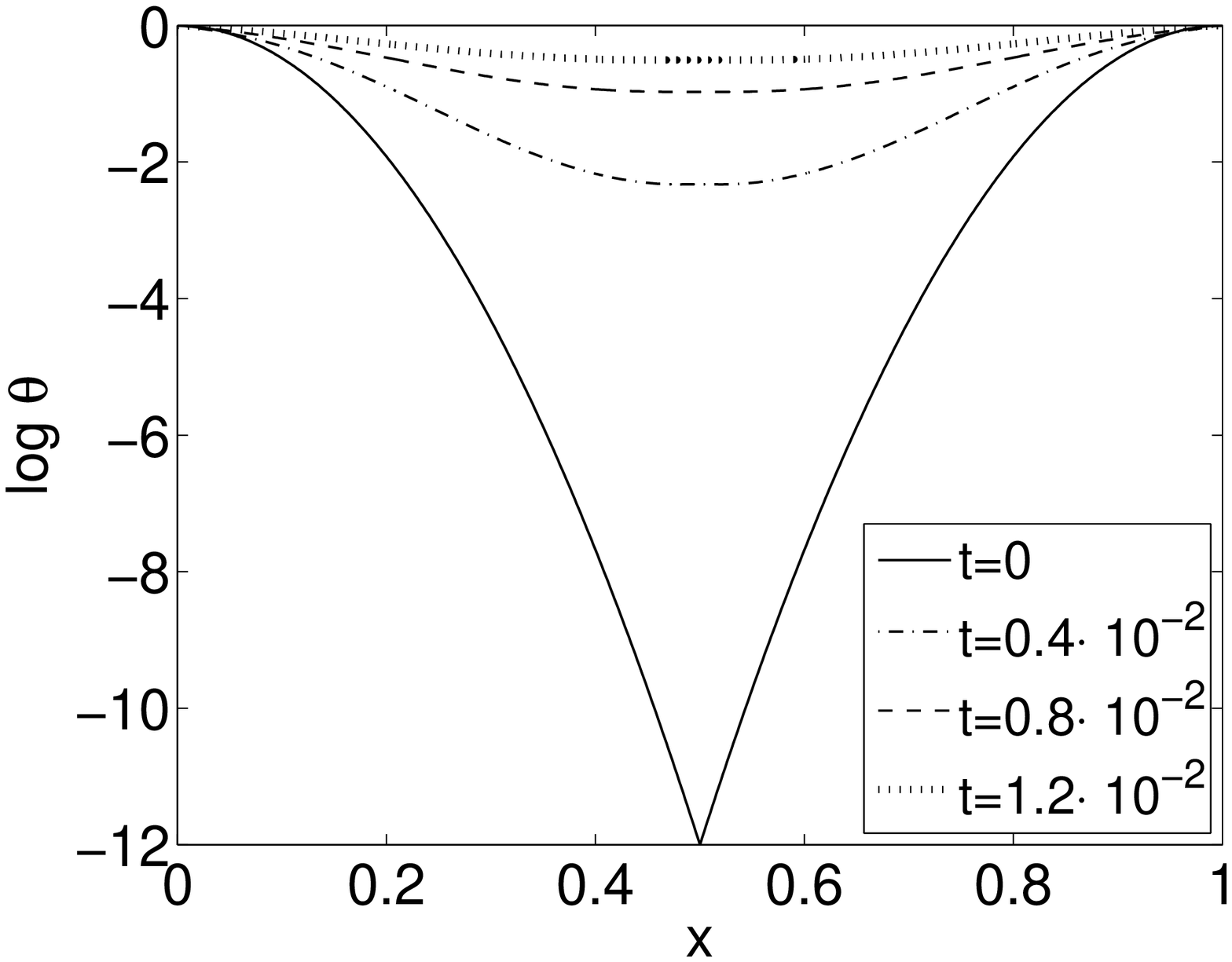}
\caption{Evolution of the particle density $n$ and the temperature $\theta$
(semi-logarithmic plot) at various times for $\beta=0.25$.}
\label{fig.betaplus025}
\end{figure}

\begin{figure}[ht]
\includegraphics[width=80mm,height=60mm]{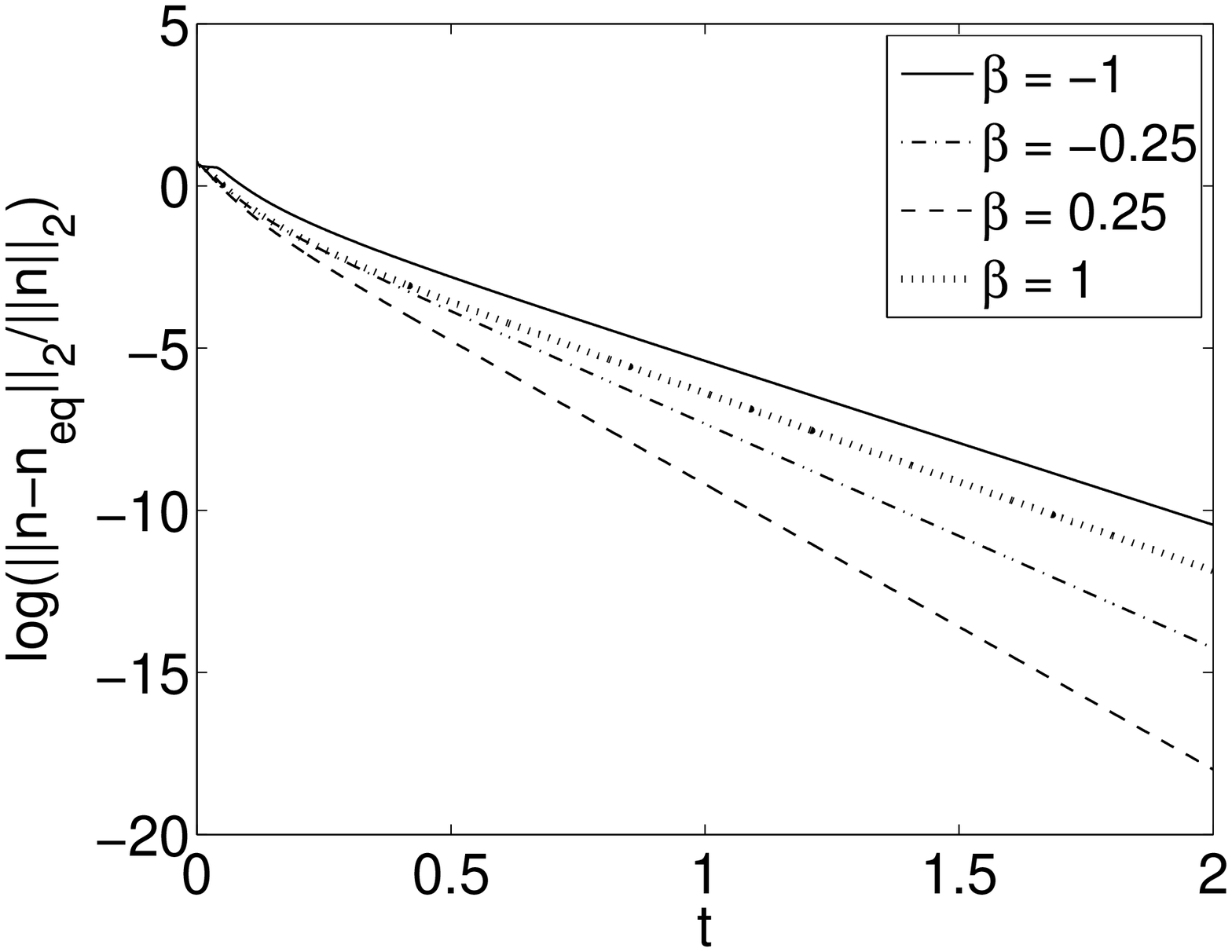}
\includegraphics[width=80mm,height=60mm]{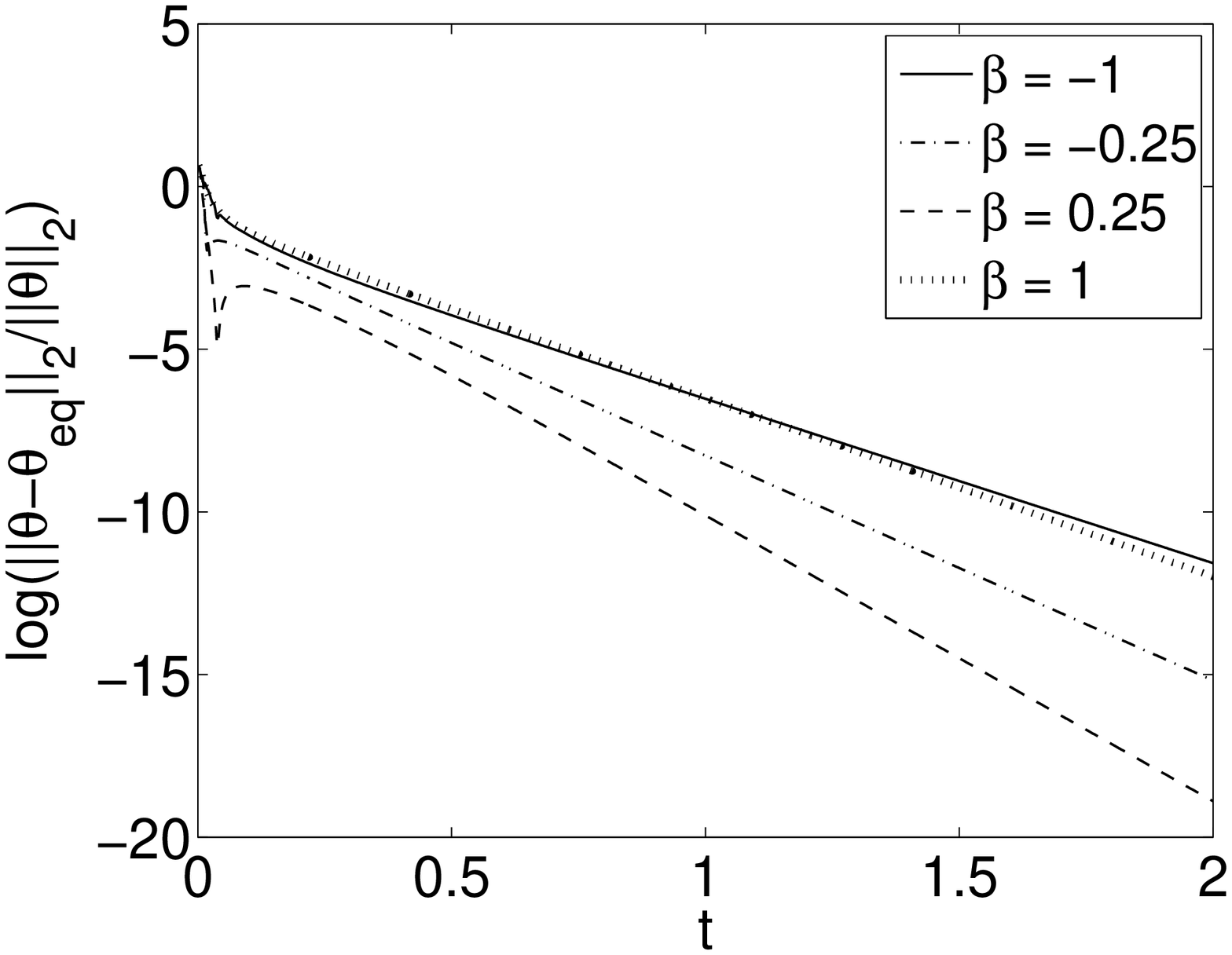}
\caption{Decay of the relative $\ell^2$ distance to the equilibrium functions
$n_{\rm eq}=1$ and $\theta_{\rm eq}=1$ for various values of $\beta$ (semi-logarithmic
plot).}
\label{fig.longtime}
\end{figure}

%%%%%%%%%%%%%%%%%%%%%%%%%%%%%%%%%%%%%%%%%%%%%%%%%%%%%%%%%%%%%%%%%%%%%%%%%%%%%

\end{document}